\documentclass[a4paper]{amsart}

\usepackage[english]{babel}
\usepackage[utf8]{inputenc}
\usepackage{amsmath}
\usepackage{amsthm}
\usepackage{amssymb}
\usepackage{graphicx}
\usepackage{enumitem}
\usepackage{tikz-cd}
\usepackage{fullpage}
\usepackage{bbm}
\usepackage{comment}
\usepackage{soul}
\usepackage{hyperref}
\setuldepth{e}

\usepackage{mathtools}

\makeatletter
\newtheorem*{rep@theorem}{\rep@title}
\newcommand{\newreptheorem}[2]{%
\newenvironment{rep#1}[1]{%
	\def\rep@title{#2~\ref{##1}}%
	\begin{rep@theorem}}%
	{\end{rep@theorem}}}
\makeatother

\usepackage[backend=biber,style=alphabetic, maxnames=100]{biblatex}
\usepackage{csquotes}

\DeclareFieldFormat{postnote}{#1}
\DeclareFieldFormat{multipostnote}{#1}

\renewcommand{\O}{\mathcal{O}}

\newcommand{\QQp}{\mathbb{Q}_p}
\newcommand{\ZZp}{\mathbb{Z}_p}

\newcommand{\ZZ}{\mathbb{Z}}

\newcommand{\FFp}{\mathbb{F}_p}
\newcommand{\FFpb}{\overline{\mathbb{F}}_p}

\newcommand{\Gal}{\mathrm{Gal}}
\newcommand{\Hom}{\mathrm{Hom}}
\newcommand{\GL}{\mathrm{GL}}
\newcommand{\ovr}[1]{\overline{#1}}

\numberwithin{equation}{section}

\theoremstyle{definition}
\newtheorem{defn}[equation]{Definition}
\newtheorem{con}[equation]{Construction}
\newtheorem{notation}[equation]{Notation}
\newtheorem*{clm}{Claim}

\theoremstyle{plain}
\newtheorem{thm}[equation]{Theorem}
\newtheorem{lem}[equation]{Lemma}
\newtheorem{prop}[equation]{Proposition}
\newreptheorem{prop}{Proposition}
\newtheorem{cor}[equation]{Corollary}

\theoremstyle{remark}
\newtheorem{rem}[equation]{Remark}

\title{Explicit Serre weights for $\operatorname{GL}_2$ via Kummer theory}
\author{Robin Bartlett \& Misja F.A. Steinmetz}
\date{24 June 2022}

\begin{document}
\begin{abstract}
	We give an explicit formulation of the weight part of Serre's conjecture for $\operatorname{GL}_2$ using Kummer theory. This avoids any reference to $p$-adic Hodge theory. The key inputs are a description of the reduction modulo $p$ of crystalline extensions in terms of certain ``$G_K$-Artin--Scheier cocycles'' and a result of Abrashkin which describes these cocycles in terms of Kummer theory.
	
	An alternative explicit formulation in terms of local class field theory was previously given by Demb\'el\'e--Diamond--Roberts in the unramified case and by the second author in general. We show that the description of Demb\'el\'e--Diamond--Roberts can be recovered directly from ours using the explicit reciprocity laws of Br\"uckner--Shaferevich--Vostokov. These calculations illustrate how our use of Kummer theory eliminates certain combinatorial complications appearing in these two papers. 
\end{abstract}

\maketitle
\setcounter{tocdepth}{1}
\tableofcontents

\section{Introduction}
\subsection*{Overview}
Serre conjectured in \cite{ser87} that every continuous irreducible odd representation $\overline{\rho}:G_{\mathbb{Q}} \rightarrow \operatorname{GL}_2(\overline{\mathbb{F}}_p)$ arose as the reduction modulo $p$ of the Galois representation attached to a modular form. Furthermore, Serre predicted the possible weights of the relevant modular forms in terms of the local representation $\overline{\rho}|_{G_{\mathbb{Q}_p}}$. As the following example illustrates, the recipe is extremely explicit. Suppose 
$$
\overline{\rho}|_{I_{\mathbb{Q}_p}} \sim \begin{pmatrix}
	\overline{\chi}_{\operatorname{cyc}}^k & c \\ 0 & 1
\end{pmatrix}, \qquad 0 \leq k \leq p-1
$$ with $\overline{\chi}_{\operatorname{cyc}}$ the mod~$p$ cyclotomic character and $I_{\mathbb{Q}_p} \subset G_{\mathbb{Q}_p}$ the inertia subgroup. Then Serre expected that $\overline{\rho}$ would be modular of weight $k+1$. The one exception is when $k=1$; in this case $\overline{\rho}$ is modular of weight $2$ if and only if the class of $c$ is \emph{peu ramifi\'e}, i.e. contained in the image of the Kummer map
$$
\mathbb{Z}_p^\times \otimes_{\mathbb{Z}_p} \overline{\mathbb{F}}_p \rightarrow H^1(G_{\mathbb{Q}_p},\FFpb(\overline{\chi}_{\operatorname{cyc}})).
$$
Otherwise $\overline{\rho}$ will be modular of weight $p+1$. 

Generalisations of this weight recipe have been made in \cite{bdj10,BLGG13} with $\mathbb{Q}$ replaced by a totally real field $F$. When $\overline{\rho}|_{G_{F_v}}$ is semisimple at each prime $v$ of $F$ dividing $p$ this is an immediate extension of Serre's description. However, the more general setup requires considerably more complicated constraints on the extension classes. The previously mentioned conjectures give a description of these extension classes in terms of reductions of crystalline representations. In a series of papers \cite{GK14,BLGG13,New14}, culminating in \cite{gls15}, these conjectures have essentially been proven. In particular, for $p>2$ and a totally real field $F$, the possible weights of a modular representation $\overline{\rho}\colon G_F \rightarrow \operatorname{GL}_2(\overline{\mathbb{F}}_p)$ can be described in terms of a set of ``local'' Serre weights $W^{\operatorname{cr}}(\overline{\rho}|_{G_{F_v}})$ defined in terms of Hodge--Tate weights of crystalline lifts of $\overline{\rho}|_{G_{F_v}}$ at places $v$ of $F$ diving $p$. This is explained in detail in Section~\ref{serreweights}.

This description of the weights in terms of crystalline lifts, while conceptually appealing, is neither explicit nor computable. The goal of this paper is to give an alternative description in the spirit of Serre's original conjecture, using the Kummer map. As a consequence, we obtain an explicit formulation of the weight part of Serre's conjecture which avoids any mention of $p$-adic Hodge theory.
\subsection*{Crystalline lifts and our main result}

To achieve this goal we are reduced to the purely local  problem of explicitly describing $W^{\operatorname{cr}}(\overline{r})$ for any continuous $\overline{r}:G_K \rightarrow \operatorname{GL}_2(\overline{\mathbb{F}}_p)$ with $K/\mathbb{Q}_p$ a finite extension. The results of \cite{gls15} give such an explicit description when $\overline{r}$ is semisimple and in general show that $W^{\operatorname{cr}}(\overline{r}) \subset W^{\operatorname{cr}}(\overline{r}^{\operatorname{ss}})$. See Section~\ref{semisimplesection}.

To state our main result let $f$ denote the residue degree of $K$ over $\mathbb{Q}_p$  and fix a uniformiser $\pi \in K$, as well as a $(p^f-1)$-th root $\pi^{1/(p^f-1)}$ in an algebraic closure. Set $L$ equal the $(p^f-1)$-th unramified extension of $K(\pi^{1/(p^f-1)})$. Then $L$ contains the primitive $p$-th roots of unity. If $l$ denotes the residue field of $L$ the Artin--Hasse exponential defines an isomorphism of $\mathbb{Z}_p$-modules
$$
vW(l)[[v]] \xrightarrow{\sim} 1+vW(l)[[v]]
$$
sending $f \mapsto \sum_{n \geq 0} \left( \frac{\varphi^n(f)}{p^n} \right)$ for $\varphi$ the $\mathbb{Z}_p$-linear endomorphism of $W(l)[[v]]$ given by $v \mapsto v^p$ and the lift of Frobenius on $W(l)$. Composing with evaluation at $v = \pi^{1/(p^f-1)}$ produces a homomorphism $vW(l)[[v]] \rightarrow 1+ \mathfrak{m}_L$ and applying $\otimes_{\mathbb{F}_p} \overline{\mathbb{F}}_p$ gives a homomorphism
$$
vl[[v]] \otimes_{\mathbb{F}_p} \overline{\mathbb{F}}_p \rightarrow L^\times \otimes_{\mathbb{Z}} \overline{\mathbb{F}}_p = H^1(G_L,\overline{\mathbb{F}}_p)
$$
with the last identification coming via Kummer theory from a fixed choice of primitive $p$-th root of unity in $L$. We extend this to a surjective homomorphism
$$
\Psi_0: l[[v]] \otimes_{\mathbb{F}_p} \overline{\mathbb{F}}_p \rightarrow L^\times \otimes_{\mathbb{Z}} \overline{\mathbb{F}}_p = H^1(G_L,\overline{\mathbb{F}}_p)
$$
by choosing any homomorphism $\psi:l \rightarrow \mathbb{Z}/p\mathbb{Z}$ and mapping $x\in l$ onto $\pi^{\psi(x)/(p^f-1)}$.
See Section~\ref{AH} for more details on these constructions. 

Then our explicit version of the weight part of Serre's conjecture is as follows.
\begin{thm}\label{intothm}
	Suppose $p>2$ and $\overline{r}:G_K \rightarrow \operatorname{GL}_2(\overline{\mathbb{F}}_p)$ is continuous with $\overline{r} = \left( \begin{smallmatrix}
		\chi_1 & c \\ 0 & \chi_2
	\end{smallmatrix}\right)$. Then there exists an explicit $\overline{\mathbb{F}}_p$-subspace of $l[[v]] \otimes_{\mathbb{F}_p} \overline{\mathbb{F}}_p$ depending on $\sigma$ and $\overline{r}^{\operatorname{ss}}|_{I_K}$ only, whose image under $\Psi_0$ we will denote by $\Psi_{\sigma}(\chi_1,\chi_2)$, such that $\sigma  \in W^{\operatorname{cr}}(\overline{r})$ if and only
	\begin{enumerate}
		\item $\sigma \in W^{\operatorname{cr}}(\overline{r}^{\operatorname{ss}})$ and
		\item $c|_{G_L} \in \Psi_{\sigma}(\chi_1,\chi_2)$.
	\end{enumerate}
\end{thm}

This result is Theorem~\ref{main} and allows us to view the subspace $\Psi_{\sigma}(\chi_1,\chi_2)$ as extending the notion of \emph{peu ramifi\'e} classes in Serre's original conjecture.

Since the map $\Psi_0$ has a large kernel, there are many possible descriptions of the subspace defining $\Psi_{\sigma}(\chi_1,\chi_2)$. For example, the results from Section~\ref{sec-refine} describe a constant $C_{\sigma} \in l[[v]] \otimes_{\mathbb{F}_p} \overline{\mathbb{F}}_p$ so that $\Psi_{\sigma}(\chi_1,\chi_2) = \Psi_0(C_{\sigma}l[[u]] \otimes_{\mathbb{F}_p} \overline{\mathbb{F}}_p)$ for $u =v^{p^f-1}$. In fact, we have the following even more explicit description. 

\begin{thm}\label{thm2}
	Assume that $\chi_1/\chi_2$ is not equal the trivial character or an unramified twist of the cyclotomic character. Let $e$ denote the ramification degree of $K/\mathbb{Q}_p$. For each $\tau \in \operatorname{Hom}_{\mathbb{F}_p}(k,\overline{\mathbb{F}}_p)$ and $n \in [0,e-1]$, there exist elements $u_{\tau,n} \in l[[v]] \otimes_{\mathbb{F}_p} \overline{\mathbb{F}}_p$ depending on $\sigma$ and $\overline{r}^{\operatorname{ss}}|_{I_K}$. Then there exists a unique pair $(J,x)$ with $J \subset \operatorname{Hom}_{\mathbb{F}_p}(k,\overline{\mathbb{F}}_p)$ and $x = (x_\tau)_{\tau:k \rightarrow \overline{\mathbb{F}}_p}$ with $x_\tau \in [0,e-1]$ such that
	$$
	\Psi_{\sigma}(\chi_1,\chi_2) = \operatorname{span}_{\overline{\mathbb{F}}_p}\left\{ \Psi_0(u_{\tau,n}) \mid n \leq \begin{cases}
		x_\tau + 1 & \text{if $\tau \circ \varphi^{-1} \in J$;} \\
		x_\tau & \text{if $\tau \circ \varphi^{-1} \not\in J$.} 
	\end{cases}\right\}
	$$
	In fact, these $\Psi_0(u_{\tau,n})$ form a basis of $\Psi_{\sigma}(\chi_1,\chi_2)$.
\end{thm}

The assumptions made on $\chi_1/\chi_2$ here are only for simplicity; for the complete statement we refer to Proposition~\ref{maximalprop} and \ref{cor-degen}. We emphasise that, even though the above theorems assert only that various subspaces or elements exist, in the body of the paper we give explicit formulae for all these objects.

\subsection*{Relation to Demb\'el\'e--Diamond--Roberts} The idea of giving a completely explicit formulation of the weight part of Serre's conjecture was first addressed in \cite{ddr16}. They assume $K/\mathbb{Q}_p$ is unramified and, in this setting, formulated a conjectural description of $W^{\operatorname{cr}}(\overline{r})$ using local class field theory to describe subspaces of $H^1(G_K,\overline{\mathbb{F}}_p(\chi_1/\chi_2))$. These predictions were subsequently proven in \cite{cegm17}. In \cite{ste22} the second author showed that when $K/\mathbb{Q}_p$ ramifies it is still possible to give an explicit description along the lines of \cite{ddr16} and prove the equivalence of this description to $W^{\operatorname{cr}}(\overline{r})$.

In each of \cite{ddr16,ste22} the relevant subspaces of $H^1(G_K,\overline{\mathbb{F}}_p(\chi_1/\chi_2))$ are described by first exhibiting a basis of $H^1(G_K,\overline{\mathbb{F}}_p(\chi_1/\chi_2))$ and then defining the subspaces as the span of certain elements of this basis. One issue with this approach is that, in certain boundary situations, deciding which basis elements should be included in the subspace requires a combinatorial recipe which is much more complicated than that in Theorem~\ref{thm2}. Even in the unramified case this recipe (see Section~\ref{sec-DDR}) is rather involved. In the presence of ramification finding a simpler and more direct description for which basis elements are to be included becomes a difficult combinatorial problem which is unlikely to have a straightforward general solution (see, for example, \cite[Ch.~7]{ste20} where simpler descriptions are given under several simplifying assumptions). One of the main motivations for this paper was to circumvent these complications.

On the other hand, one can wonder whether Theorem~\ref{intothm} could be used to recover the results of \cite{ddr16} in the unramified case. We do this in the last part of the paper, using the explicit reciprocity law of Br\"uckner--Shafervich--Vostokov to pass between our Kummer theoretic description and that given in terms of local class field theory.

\begin{prop}\label{prop:intro-comparison}
	Assume $K/\mathbb{Q}_p$ is unramified and let $L^{\operatorname{DDR}}_{\sigma}(\chi_1,\chi_2) \subset H^1(G_K,\FFpb(\chi_1/\chi_2))$ denote the subspace defined in \cite{ddr16} (see Section~\ref{sec-DDR} for more details). Then $$
	L^{\operatorname{DDR}}_{\sigma}(\chi_1,\chi_2) = \Psi_{\sigma}(\chi_1,\chi_2)^{\operatorname{Gal}(L/K) = \chi_2/\chi_1}$$
	under the identification $H^1(G_K,\FFpb(\chi_1/\chi_2)) = H^1(G_L,\FFpb)^{\operatorname{Gal}(L/K) = \chi_2/\chi_1}$.
\end{prop}

Of course this proposition follows immediately given that both subspaces $\Psi_{\sigma}(\chi_1,\chi_2)^{\operatorname{Gal}(L/K) = \chi_2/\chi_1}$ and $L^{\operatorname{DDR}}_{\sigma}(\chi_1,\chi_2)$ have the same interpretation in terms of crystalline lifts.  However, in the spirit of this entire paper, our calculations avoid any $p$-adic Hodge theoretic description. In particular, our calculations give an alternative proof of the results in \cite{ddr16} when $p>2$. We believe it is possible to use a strategy similar to the one we have used here to give a direct comparison of the results of this paper to the results of \cite{ste22} in the ramified case.

\subsection*{Method of proof}
To prove Theorem~\ref{intothm} we need to show that, if $\overline{r}$ admits a crystalline lift with Hodge--Tate weights corresponding to $\sigma$, then this imposes significant conditions on $\overline{r}$ which can ultimately be formulated in terms of the Artin--Hasse exponential. This is done in three steps:

\subsubsection*{Step 1} This uses the integral $p$-adic Hodge theory developed in \cite{gls15}. If $r$ is a crystalline lift of $\overline{r}$ witnessing $\sigma \in W^{\operatorname{cr}}(\overline{r})$, then \cite{gls15} gives a description, in terms of $\sigma$, of the shape of the reduction modulo $p$ of the Breuil--Kisin module $\overline{\mathfrak{M}}$ associated to $r$. Here $\overline{\mathfrak{M}}$ is a finite free $k[[u]] \otimes_{\mathbb{F}_p} \overline{\mathbb{F}}_p$-module equipped with a semi-linear Frobenius endomorphism, and \cite{gls15} describes the matrix of this endomorphism in terms of a particular choice of basis (see Proposition~\ref{gls}). Set $K_\infty =K(\pi^{1/p^\infty})$ for $\pi^{1/p^\infty}$ a compatible system of $p$-th power roots of $\pi \in K$. Then $\overline{\mathfrak{M}}$ and $\overline{r}|_{G_{K_\infty}}$ are related using the existence of a $\varphi,G_{K_\infty}$-equivariant identification
$$
\overline{\mathfrak{M}} \otimes_{k[[u]]} C^\flat = \overline{r}^\vee \otimes_{\mathbb{F}_p} C^\flat,
$$
where $C^\flat$ denotes a specific algebraic closure of $k((u))$ and $G_{K_\infty}$ acts trivially on $\overline{\mathfrak{M}}$. In particular, $\overline{r}^\vee|_{G_{K_\infty}} = (\overline{\mathfrak{M}} \otimes_{\mathbb{F}_p} C^\flat)^{\varphi=1}$. Concretely, if $\beta$ is a basis of $\overline{\mathfrak{M}}$ and $\alpha = \beta D$ generates $\overline{r}^\vee$ (so that $D$ is a matrix such that $\varphi(D^{-1})D$ equals the matrix of the Frobenius relative to $\beta$) then the $G_{K_\infty}$-action on $\alpha$ is given by
$$
\sigma(\alpha) = \alpha \sigma(D)D^{-1}.
$$
From this we deduce a statement of the following shape: if $\overline{r} \sim \left( \begin{smallmatrix}
	\chi_1 & c\\ 0 & \chi_2
\end{smallmatrix}\right)$, then there exists a subspace $\Psi^{\operatorname{AS}}_{\sigma}(\chi_1,\chi_2) \subset H^1(G_{K_\infty}, \FFpb(\chi_1/\chi_2))$ defined in terms of Artin--Scheier cocycles so that $\sigma \in W^{\operatorname{cr}}(\overline{r})$ implies
$$
c|_{G_{K_\infty}} \in \Psi_{\sigma}^{\operatorname{AS}}(\chi_1,\chi_2).
$$
We emphasise that everything so far follows, more or less, immediately from \cite{gls15}. In the unramified case this is the essential tool used to prove the conjecture of \cite{ddr16} in \cite{cegm17}.

\subsubsection*{Step 2}The second step is to upgrade the description of the $G_{K_\infty}$-action on $\overline{r}$, given in terms of $\overline{\mathfrak{M}}$, to a description of the $G_K$-action. For this we first recall that the action of $G_{K_\infty}$ on $C^\flat$ naturally extends to a $G_K$-action. Therefore, $C^\flat$-semi-linearly extending the $G_K$-action on $\overline{r}^\vee \otimes_{\mathbb{F}_p} C^\flat$ we obtain a $\varphi$-equivariant $G_K$-action on $\overline{\mathfrak{M}} \otimes_{k[[u]]} C^\flat$. Since the $G_K$-action on $\overline{r}$ comes from the reduction modulo $p$ of a crystalline representation this $G_K$-action must satisfy the following divisibility
$$
\sigma(m)-m \in \mathfrak{M} \otimes_{k[[u]]} u^{(e+p-1)/(p-1)}\mathcal{O}_{C^\flat},
$$
for all $\sigma \in G_K$ and $m \in \mathfrak{M}$.

On the other hand, ideas from \cite{B21} give a procedure which, in good cases, constructs an alternative $G_K$-action on $\overline{\mathfrak{M}} \otimes_{k[[u]]} C^\flat$. This is done as follows: choose a basis $\beta$ of $\overline{\mathfrak{M}}$ and define a ``naive'' $G_K$-action $\sigma_{\operatorname{naive},\beta}$ on $\overline{\mathfrak{M}} \otimes_{k[[u]]} C^\flat$ by semi-linearly extending the action which fixes $\varphi(\beta)$. In general, this action will not be $\varphi$-equivariant. However, one can attempt to produce a $\varphi$-equivariant action from it by considering 
$$
\sigma = \operatorname{lim}_{n \rightarrow \infty} \varphi^n \circ \sigma_{\operatorname{naive},\beta} \circ \varphi^{-n}.
$$
Typically (i.e. for an arbitrary Breuil--Kisin module) this limit will not converge. However, in our case this limit really exists, due to the special shape of $\overline{\mathfrak{M}}$ (and ultimately the fact that the Hodge--Tate weights of $r$ are sufficiently small). Furthermore, one shows that this is the unique $G_K$-action satisfying the above divisibility. Therefore, the $G_K$-action computed by this limit coincides with the $G_K$-action coming from $\overline{r}$. 

Concretely, if $\alpha = \beta D$ is a basis of $\overline{r}^\vee$ then the $G_K$-action on $\overline{r}^\vee$ is given by
$$
\sigma(\alpha) = \alpha \left( \operatorname{lim}_{n \rightarrow \infty} \varphi^n(\sigma(D)D^{-1}) \right).
$$ 
This allows us to reformulate the implication in the final part of Step (1); we obtain that $\sigma \in W^{\operatorname{cr}}(\overline{r})$ implies
$$
c \in \Psi_{\sigma}^{G_K\operatorname{-AS}}(\chi_1,\chi_2),
$$
where now $\Psi_{\sigma}^{G_K\operatorname{-AS}}(\chi_1,\chi_2) \subset H^1(G_K,\FFpb(\chi_1/\chi_2))$ is a subspace of certain ``$G_K$-Artin--Schreier'' coycles. 

\subsubsection*{Step 3} The final step is to produce a dictionary between the restriction of these``$G_K$-Artin--Schreier'' cocycles to $G_L$ and Kummer cocycles. This was done in a beautiful computation of Abrashkin \cite{abr97}. To be precise he considers any $h \in vl[[v]]$ and chooses $h' \in C^\flat$ so that $\varphi(h')-h' = h$. Then he considers the ``$G_L$-Artin--Schreier'' cocycle $G_L \rightarrow \mathbb{F}_p$ defined by 
$$
\sigma  \mapsto \operatorname{lim}_{n\to \infty} \varphi^n \left( \sigma\left(h'\right)- h'\right).
$$
Equivalently, this cocycle can be described as sending $\sigma  \in G_L$ onto the image of $ \sigma\left(h'\right)- h'$ under the map $\mathcal{O}_{C^\flat} \rightarrow \overline{\mathbb{F}}_p$. The restriction to $G_L$ of those cocycles in  $\Psi_{\sigma}^{G_K\operatorname{-AS}}(\chi_1,\chi_2)$ all have this form. Abrashkin gives an explicit formula, using the map $\Psi$ from above, which expresses this cocycle as a Kummer cocycle (see Proposition~\ref{Abr}). Using this formula we obtain a Kummer theoretic description the restriction of $\Psi_\sigma^{G_K-\operatorname{AS}}(\chi_1,\chi_2)$ to $G_L$ in terms of an explicit $\Psi_\sigma(\chi_1,\chi_2)$. We deduce, for $\overline{r} \sim \left( \begin{smallmatrix}
	\chi_1 & c \\ 0 & \chi_2
\end{smallmatrix}\right)$, that 
$\sigma \in W^{\operatorname{cr}}(\overline{r})$ implies
$$
c|_{G_L} \in \Psi_{\sigma}(\chi_1,\chi_2).
$$  
It only remains to prove the opposite implication: if $c|_{G_L} \in \Psi_\sigma(\chi_1,\chi_2)$ then we must produce a crystalline lift of $\overline{r}$ witnessing $\sigma \in W^{\operatorname{cr}}(\overline{r})$. We do this in the standard way by producing crystalline lifts of the characters $\chi_1$ and $\chi_2$ and then considering the image in $H^1(G_K,\FFpb(\chi_1/\chi_2))$ of the space of crystalline extensions of these two lifts. By the above this image is contained in $\Psi_{\sigma}(\chi_1,\chi_2)$ and we will be done if these two subspaces are equal. This follows by comparing dimensions.

\subsection*{Acknowledgements} Both authors would like to thank Fred Diamond for introducing them to the fascinating topic of Serre weights. The first author would like to thank the Max Planck Institute for Mathematics, Bonn for support during the early part of this project. The first author was also funded by the Deutsche Forschungsgemeinschaft (DFG, German Research Foundation) under Germany's Excellence Strategy EXC 2044 –390685587, Mathematics Münster: Dynamics–Geometry–Structure. The second author would like to thank the Mathematical Institute at Leiden University for their continued support especially during the COVID-19 pandemic periods in 2020 and 2021.

\section{Serre weights}\label{serreweights}

Throughout $K/\mathbb{Q}_p$ is a finite extension with residue field $k$. Set $f = [k:\mathbb{F}_p]$ and $e = e(K/\mathbb{Q}_p)$ the ramification degree. Choose a uniformiser $\pi \in K$ and a $(p^f-1)$-th root $\pi^{1/(p^f-1)}$ in a completed algebraic closure $C$ of $K$. Set $L$ equal the unramified extension of $K(\pi^{1/p^f-1})$ of degree $p^f-1$. Write $l$ for the residue field of $L$.
\begin{defn}
	A Serre weight (for $\operatorname{GL}_2(k)$) is an isomorphism class of irreducible $\overline{\mathbb{F}}_p$-representations of $\operatorname{GL}_2(k)$. Any such class can be represented by
	$$
	\sigma_{a,b} := \bigotimes_{\tau \in \operatorname{Hom}_{\mathbb{F}_p}(k,\overline{\mathbb{F}}_p)} \left( \operatorname{det}^{b_\tau} \otimes_{k} \operatorname{Sym}^{a_\tau - b_\tau} k^2 \right) \otimes_{k,\tau} \overline{\mathbb{F}}_p
	$$
	for uniquely determined integers $a_\tau,b_\tau$ satisfying $b_\tau, a_\tau - b_\tau \in [0,p-1]$ and not all $b_\tau$ equal to $p-1$.
\end{defn}

Suppose $V$ is a Hodge--Tate representation of $G_K$ on a $\overline{\mathbb{Q}}_p$-vector space. For each $\kappa \in \operatorname{Hom}_{\mathbb{Q}_p}(K,\overline{\mathbb{Q}}_p)$, the $\kappa$-Hodge--Tate weights $\operatorname{HT}_\kappa(V)$ of $V$ is the multiset of integers which contains $i$ with multiplicity
$$
\operatorname{dim}_{\overline{\mathbb{Q}}_p}(V \otimes_{\kappa,K} C(-i))^{G_K}.
$$ 
Here $C(i)$ is a completed algebraic closure of $K$ with the twisted $G_K$-action $\sigma(a) = \chi_{\operatorname{cyc}}(\sigma)^i \sigma(a)$ for $\chi_{\operatorname{cyc}}$ the $p$-adic cyclotomic character. Thus, $\operatorname{HT}_{\kappa}(\chi_{\operatorname{cyc}}) = \lbrace 1 \rbrace$ for every $\kappa$.

\begin{defn}\label{HTwt}
	A lift of a Serre weight $\sigma = \sigma_{a,b}$ is a tuple of pairs of integers $\widetilde{\sigma} = (\widetilde{a}_\kappa,\widetilde{b}_\kappa)_{\kappa \in \operatorname{Hom}_{\mathbb{Q}_p}(K,\overline{\mathbb{Q}}_p)}$ such that, for each $\tau: k \rightarrow \overline{\mathbb{F}}_p$, there is an indexing
	\begin{equation}\label{index}
		\lbrace \kappa \in \operatorname{Hom}_{\mathbb{Q}_p}(K,\overline{\mathbb{Q}}_p) \mid \kappa|_k = \tau \rbrace = \lbrace \tau_0,\ldots,\tau_{e-1} \rbrace
	\end{equation}
	so that
	$$
	(\widetilde{a}_\kappa,\widetilde{b}_\kappa) = \begin{cases}
		(a_\tau+1,b_\tau) & \text{if $\kappa = \tau_0$;} \\
		(1,0) & \text{if $\kappa = \tau_i$ for $i >0$.}
	\end{cases}
	$$
	We say that a crystalline representation of $G_K$ on a finite free $\overline{\mathbb{Z}}_p$-module $V$ has Hodge type $\sigma$ if there exists a lift $\widetilde{\sigma} = (\widetilde{a}_\kappa,\widetilde{b}_\kappa)$ so that
	$$
	\operatorname{HT}_\kappa(V) = (\widetilde{a}_\kappa,\widetilde{b}_\kappa)
	$$
	for every $\kappa:K\rightarrow \overline{\mathbb{Q}}_p$.
\end{defn}
\begin{defn}
	For a continuous $\overline{r}: G_K \rightarrow \operatorname{GL}_2(\overline{\mathbb{F}}_p)$, we let $W^{\operatorname{cr}}(\overline{r})$ denote the set of Serre weights $\sigma_{a,b}$ for which there exists a crystalline representation of $G_K$ on a finite free $\overline{\mathbb{Z}}_p$-module $V$ with Hodge type $\sigma_{a,b}$ and $V \otimes_{\overline{\mathbb{Z}}_p} \overline{\mathbb{F}}_p \cong \overline{r}$.
\end{defn}

The following motivates the definition of $W^{\operatorname{cr}}(\overline{r})$. Suppose that $F$ is a totally real extension of $\mathbb{Q}$. Let $\overline{\rho}: G_F \rightarrow \operatorname{GL}_2(\mathbb{F})$ be a continuous and absolutely irreducible representation which arises as the reduction modulo~$p$ of a $p$-adic representation associated to a Hilbert modular eigenform of parallel weight $2$. For each place $v$ of $F$ dividing $p$, let $k_v$ denote the residue field of $F_v$. Let $D$ be a quarternion algebra with centre $F$ and which is split at all places dividing $p$, and at zero or one infinite place. In \cite[4.3.3]{GK14} it is explained what it means for $\overline{\rho}$ to be modular for $D$ of weight $\sigma = \otimes_{v\mid p} \sigma_v$ (each $\sigma_v$ being a Serre weight for $\operatorname{GL}_2(k_v)$). 

\begin{thm}
	Suppose that $p>2$. Assume also that $\overline{\rho}$ is modular, compatible with $D$ in the sense of \cite[4.3.4]{GK14}, and that $\overline{\rho}|_{G_{F(\zeta_p)}}$ is irreducible. If $p=5$ assume that the projective image of $\overline{\rho}|_{G_{F(\zeta_p)}}$ is not isomorphic to $A_5$. Then $\overline{\rho}$ is modular for $D$ of weight $\sigma = \otimes_{v\mid p} \sigma_v$ if and only if $\sigma_v \in W^{\operatorname{cr}}(\overline{\rho}|_{G_{F_v}})$ for each $v\mid p$.
\end{thm} 
\begin{proof}
	See \cite[\S 4.1 and \S4.2]{gls15}.
\end{proof}
\section{Explicit Serre weights in the semisimple case}\label{semisimplesection}

In this section we recall explicit descriptions of $W^{\operatorname{cr}}(\overline{r})$ when $\overline{r}$ is semisimple. Recall our fixed choice of uniformiser $\pi \in K$ as well as the $(p^f-1)$-th root $\pi^{1/(p^f-1)}$. Using this we can define a character 
$$
\omega:G_K\rightarrow k^\times
$$
by setting $\omega(\sigma)$ equal the image of $\sigma(\pi^{1/(p^f-1)})/ \pi^{1/(p^f-1)}$ in $k^\times$. Note that $\omega$ depends upon $\pi^{1/(p^f-1)}$, but its restriction to the inertia subgroup $I_K$ does not. For $\tau:k \rightarrow \overline{\mathbb{F}}_p$, set $\omega_\tau := \tau \circ \omega$. Then, for every character $\chi:G_K\rightarrow \overline{\mathbb{F}}_p^\times$, one can write 
$$\chi|_{I_K} = \prod_\tau \omega_\tau^{n_\tau}
$$
for some integers $n_\tau$. The $n_\tau$ are uniquely determined if we further ask that $n_\tau \in [1,p]$ and not every $n_\tau$ equals $p$. Notice also that, since $\omega_{\tau}^p = \omega_{\tau \circ \varphi}$, we can write $\chi|_{I_K} = \omega_\tau^{\Omega_{\tau, n}}$ for any $\tau:k \rightarrow \overline{\mathbb{F}}_p$, where
\begin{equation}\label{omegadef}
	\Omega_{\tau, n} := \sum_{i=0}^{f-1} p^i n_{\tau \circ \varphi^i}.
\end{equation}

\begin{defn}\label{exp}
	Let $\overline{r}: G_K \rightarrow \operatorname{GL}_2(\overline{\mathbb{F}}_p)$ be continuous and semisimple. Following \cite[4.1]{gls15} we define a set of Serre weights $W^{\operatorname{exp}}(\overline{r})$ as follows:
	\begin{itemize}
		\item If $\overline{r}$ is a direct sum of two characters, then $\sigma_{a,b} \in W^{\operatorname{exp}}(\overline{r})$ if there exists
		$$
		J \subset \operatorname{Hom}_{\mathbb{F}_p}(k,\overline{\mathbb{F}}), \qquad x_\tau \in [0,e-1] \text{ for each $\tau \in  \operatorname{Hom}_{\mathbb{F}_p}(k,\overline{\mathbb{F}}_p)$}
		$$ so that
		$$
		\overline{r}|_{I_K} \cong \begin{pmatrix} 
			\prod_{\tau \in J} \omega_{\tau}^{a_\tau + 1 + x_{\tau}} \prod_{\tau \not\in J} \omega_{\tau}^{b_\tau + x_{\tau}} & 0 \\ 
			0 & \prod_{\tau \not\in J} \omega_{\tau}^{a_\tau + e - x_{\tau}} \prod_{\tau \in J} \omega_{\tau}^{b_\tau + e-1 - x_{\tau}}
		\end{pmatrix}.
		$$
		\item If $\overline{r}$ is irreducible, then $\sigma_{a,b} \in W^{\operatorname{exp}}(\overline{r})$ if and only if there exists 
		$$
		J \subset \operatorname{Hom}_{\mathbb{F}_p}(k_2,\overline{\mathbb{F}}), \qquad x_\tau \in [0,e-1] \text{ for each $\tau \in  \operatorname{Hom}_{\mathbb{F}_p}(k,\overline{\mathbb{F}}_p)$}
		$$
		so that
		$$
		\overline{r}|_{I_K} \cong \begin{pmatrix} 
			\prod_{\tau \in J} \omega_{\tau}^{a_\tau + 1 + x_{\tau|_k}} \prod_{\tau \not\in J} \omega_{\tau}^{b_\tau + x_{\tau|_k}} & 0 \\ 
			0 & \prod_{\tau \not\in J} \omega_{\tau}^{a_\tau + 1 + e-1+ x_{\tau|_k}} \prod_{\tau \in J} \omega_{\tau}^{b_\tau + e-1- x_{\tau|_k}}
		\end{pmatrix}
		$$ 
		and so that $\operatorname{Hom}_{\mathbb{F}_p}(k_2,\overline{\mathbb{F}}_p)$ is the disjoint union of $J$ and $\{\tau \circ \sigma \mid \tau \in J\}$, where $\sigma$ denotes the non-trivial element of $\operatorname{Gal}(k_2/k)$; here $k_2$ denotes the unique degree $2$ extension of $k$.
	\end{itemize}
\end{defn}

\begin{thm}[Gee--Liu--Savitt, Wang]\label{semisimple}
	If $\overline{r}$ is semisimple, then $W^{\operatorname{exp}}(\overline{r}) = W^{\operatorname{cr}}(\overline{r})$.
\end{thm}
\begin{proof}
	When $p>2$ this is \cite[5.1.5]{gls15} except with $x_\tau$ from loc. cit. replaced with $e-1-x_\tau$ for $\tau \not\in J$ (this renormalisation is for convenience later on). When $p=2$ the methods of Gee--Liu--Savitt have been adapted by Wang (see \cite[Theorem~5.4]{wan17}).
\end{proof}

When $\overline{r}$ is reducible but not semisimple \cite{gls15} shows that  $W^{\operatorname{cr}}(\overline{r}) \subset W^{\operatorname{exp}}(\overline{r}^{\operatorname{ss}})$. However, this inclusion is rarely an equality. Our goal is to give a an explicit condition on the extension class of $\overline{r}$ which determines whether $\sigma_{a,b} \in W^{\operatorname{exp}}(\overline{r}^{\operatorname{ss}})$ is contained in $W^{\operatorname{cr}}(\overline{r})$.

\section{The Artin--Hasse exponential}\label{AH}

In this section we use the Artin--Hasse exponential to construct certain subspaces of $H^1(G_L,\overline{\mathbb{F}}_p)$ from power series in $l[[v]] \otimes_{\mathbb{F}_p} \overline{\mathbb{F}}_p$. Later (in Section~\ref{secKummer}) we will use these subspaces to define versions of $W^{\operatorname{exp}}(\overline{r})$ for non-semisimple $\overline{r}$. In order to apply results from \cite{vos79} and \cite{abr97} we assume that $p>2$ in this section.
\begin{con}\label{EAH}
	Recall that $l$ denotes the residue field of $L$. Then \cite[Proposition 1]{vos79} produces an isomorphism of $\mathbb{Z}_p$-modules
	$$
	E^{\operatorname{AH}}: vW(l)[[v]] \xrightarrow{\sim} 1 + vW(l)[[v]],
	$$
	given by
	$$
	x \mapsto \operatorname{exp}\left( \sum_{n \geq 0} (\frac{\varphi}{p})^n(x) \right).
	$$
	Here $\operatorname{exp}(x)=\sum_{i \geq 0} \frac{x^i}{i!}$ and $\varphi$ denotes the $\mathbb{Z}_p$-linear operator on $W(l)[[v]]$, which acts as the Witt vector Frobenius on $W(l)$ and which sends $v \mapsto v^p$. Applying $\otimes_{\mathbb{Z}_p} \mathbb{F}_p$ produces an isomorphism $vl[[v]] \xrightarrow{\sim} (1+vW(l)[[v]]) \otimes_{\mathbb{Z}_p} \mathbb{F}_p$.
	We can extend $\overline{E}^{\operatorname{AH}}$ to a surjective homomorphism 
	$$
	\overline{E}^{\operatorname{AH}}: l[[v]] \rightarrow W(l)((v))^\times \otimes_{\mathbb{Z}} \mathbb{F}_p 
	$$
	by choosing any surjective group homomorphism $\psi:l \rightarrow \mathbb{Z}/p\mathbb{Z}$ and setting $\overline{E}^{\operatorname{AH}}(x) = v^{\psi(x)}$ for $x \in l$.
\end{con}

% 	\footnote{To see that $E^{\operatorname{AH}}$ does indeed take values in $1+vW(l)[[v]]$ write $x = \sum_{i,j \geq 0} [x_{ij}]p^iv^j$ with $x_{ij} \in l$. Then
	% 		$$
	% 		\operatorname{exp}\left( \sum_{n \geq 0} (\frac{\varphi}{p})^n(x)\right) = \prod_{i,j \geq 0} \operatorname{exp}\left( \sum_{n \geq 0} \frac{([x_{ij}]v^j)^{p^n}}{p^n} \right)^{p^i}
	% 		$$ 
	% 		and the claimed integrality follows from that of the Artin--Hasse exponential $\operatorname{exp}(\sum_{n \geq 0} \frac{v^{p^n}}{p^n}) \in 1 + v \mathbb{Z}_p[[v]]$.
	% 	} 
Composing $\overline{E}^{\operatorname{AH}}$ with evaluation at $v= \pi^{1/(p^f-1)}$ produces a homomorphism
\begin{equation}\label{map}
	\Psi_0: l[[v]] \rightarrow L^\times \otimes_{\mathbb{Z}} \mathbb{F}_p = H^1(G_L,\mu_p(L))	
\end{equation}
Here $\mu_p(L)$ denotes the group of $p$-th roots of unity in $L$ and the identification on the right is given by the Kummer map (notice that by construction $L$ contains a primitive $p$-th root of unity).

	\begin{rem}
    The reason for making the somewhat artificial extension of $\overline{E}^{\operatorname{AH}}$ from $vl[[v]]$ to $l[[v]]$ is that it allows us to give a uniform statement of our main theorem. In all but one case we will only need to view $\overline{E}^{\operatorname{AH}}$ (or $\Psi_0$) as a function on $vl[[v]]$. To incorporate the one degenerate case however, it is necessary to have a map surjecting onto $H^1(G_L,\mu_p(L))$; for this reason we ask that the constant terms in $l[[v]]$ be mapped onto powers of $v$. See Lemma~\ref{lem-degen} and Corollary~\ref{cor-degen} for more precise results regarding this degenerate case.
\end{rem}

The motivation for considering $\Psi_0$ comes a result of Abrashkin \cite{abr97}, which we will explain now.

\begin{defn}
	To state Abrashkin's result we write $\mathcal{O}_C$ for the ring of integers in the completed algebraic closure $C$ of $K$, and $\mathcal{O}_{C^\flat} := \varprojlim_{x\mapsto x^p} \mathcal{O}_C/p$ for its tilt. Recall that $\mathcal{O}_{C^\flat}$ multiplicatively identifies with $\varprojlim_{x\mapsto x^p} \mathcal{O}_C$. Fix a choice of compatible system $\pi^{1/(p^f-1)p^\infty} \in C$  of $p$-th power roots of $\pi^{1/(p^f-1)}$ so that 
	$$
	(\pi^{1/(p^f-1)},\pi^{1/p(p^f-1)}, \pi^{1/p^2(p^f-1)},\ldots) \in \mathcal{O}_{C^\flat}
	$$
	Then $v \mapsto (\pi^{1/(p^f-1)},\pi^{1/p(p^f-1)}, \pi^{1/p^2(p^f-1)},\ldots)$ defines an embedding
	$$
	l[[v]] \hookrightarrow \mathcal{O}_{C^\flat}
	$$
	via which we view $\mathcal{O}_{C^\flat}$ as an $l[[v]]$-algebra. Set $u := v^{p^f-1}$. Thus, we also get an embedding $k[[u]] \hookrightarrow \mathcal{O}_{C^\flat}$. The ring $\mathcal{O}_{C^\flat}$ is $u$-adically complete and $C^\flat := \mathcal{O}_{C^\flat}[\frac{1}{u}] = \operatorname{Frac}\mathcal{O}_{C^\flat}$ is algebraically closed.
\end{defn}

Note that $G_K$-acts naturally on $\mathcal{O}_{C^\flat}$ via its action on $\mathcal{O}_C$. Under this action the subrings $l[[v]]$ and $k[[v]]$ are not $G_K$-stable. However, they are stable under the action of $G_{K_\infty}$ for $K_\infty =K(\pi^{1/p^\infty})$. Furthermore, this action factors through the surjection $G_{K_\infty} \rightarrow \operatorname{Gal}(L/K)$ (which exists because $L \cap K_\infty =K$ since $L/K$ is tamely ramified while $K_\infty/K$ is totally wildly ramified) and is concretely given by
\begin{equation}\label{G(L/K)-action}
	g \cdot \sum_{i \geq 0} f_i v^i = \sum_{i\geq 0} g(f_i) \omega(g)^i v^i
\end{equation}
for the character $\omega$ defined in Section~\ref{semisimplesection}.

\begin{thm}[Abrashkin]\label{Abr}
	Fix a generator $\epsilon_1 \in \mu_p(L)$ and choose $z(v) \in W(l)[[v]]$ with $z(\pi^{1/(p^f-1)}) = \epsilon_1$. Suppose that $h \in u^{-ep/(p-1)} vl[[v]]$. Then there exists a homomorphism $c_h:G_L \rightarrow \mathbb{F}_p$ such that: 
	\begin{itemize}
		\item If $H \in C^\flat$ satisfies $H^p - H = h$ then $g(H) - H \equiv c_h(g)$ modulo $\mathfrak{m}_{C^\flat}$ for every $g \in G_L$.
		\item The image of $h(z(v)^p-1) \in vl[[v]]$ under $\Psi_0$ from \eqref{map} is given by $g \mapsto \epsilon_1^{c_h(g)}$ for all $g \in G_L$.
	\end{itemize}
	In the second bullet point we write $z(v)$ for its image in $l[[v]]$ and use that $z(v)-1 \in u^{e/(p-1)}l[[v]]$ (which follows from the fact that $\epsilon_1-1$ has $p$-adic valuation $1/(p-1)$).
\end{thm}
\begin{proof}
	This is the main lemma in \cite[Section 2.3]{abr97} specialised to the case $M=1$. For the benefit of the reader let us reiterate why there exists a function $c_h: G_L \rightarrow \mathbb{F}_p$ so that
	$$
	g(H) - H \equiv c_h(g) \text{ modulo }\mathfrak{m}_{C^\flat}.
	$$
	We claim that it is enough to check that $g(h) - h \in \mathfrak{m}_{C^\flat}$. Indeed, if this is the case then any solution to $X^p-X = g(h)-h$ can be written as $X_0 + X_1$ with $X_0 \in \mathbb{F}_p$ and $X_1 \in \mathfrak{m}_{C^\flat}$. If we take $X = g(H)-H$ then we can set $c_h(g) = X_0$. 
	
	To check that $g(h) - h \in \mathfrak{m}_{C^\flat}$ notice that, since $h \in u^{-ep/(p-1)}vl[[v]]$, it is enough to check that $g(v^i) -v^i \in v^{i}u^{ep/(p-1)} \mathcal{O}_{C^\flat}$ for any $i \in \mathbb{Z}$. To see this write $g(v^i) - v^i = v^i(\zeta^i -1)$ for some element $\zeta = (1,\zeta_1,\zeta_2, \ldots) \in \mathcal{O}_{C^\flat}$. Then recall the well-known fact that  $\zeta -1$ generates the ideal $u^{ep/(p-1)}\mathcal{O}_{C^\flat}$ (see, for example, \cite[5.1.3]{Fon94}).
\end{proof}

As a first application of this result we produce a refinement of \eqref{map} by defining
$$
\Psi: u^{-ep/(p-1)}l[[v]]\otimes_{\mathbb{F}_p} \overline{\mathbb{F}}_p  \rightarrow H^1(G_L,\overline{\mathbb{F}}_p)
$$
as the $\overline{\mathbb{F}}_p$-linear extension of the composite
$$
u^{-ep/(p-1)} l[[u]] \xrightarrow{z(v)^p-1)} l[[v]] \xrightarrow{\Psi_0} H^1(G_L,\mu_p(L)) \rightarrow H^1(G_L,\mathbb{F}_p),
$$
where $z(v)$ is as in Theorem~\ref{Abr} and the last map is the identification induced by the choice of $\epsilon_1$. In other words, the last map sends $f: G_L \rightarrow \mu_p(L)$ onto the homomorphism $c:G_L \rightarrow \mathbb{F}_p$ characterised by $f(g) = \epsilon_1^{c(g)}$.

In the following corollary we write $\varphi$ for the $\overline{\mathbb{F}}_p$-linear extension of the $p$-th power map on $\mathcal{O}_{C^\flat}$.
\begin{cor}\label{Abrcor}\hfill
	\begin{enumerate}
		\item If $H \in C^\flat \otimes_{\mathbb{F}_p} \overline{\mathbb{F}}_p$ satisfies $\varphi(H) - H = h$ for $h \in u^{-ep/(p-1)}vl[[v]] \otimes_{\mathbb{F}_p} \overline{\mathbb{F}}_p$, then
		$$
		g(H) - H \equiv \Psi(h)(g) \text{ modulo }\mathfrak{m}_{C_\flat} \otimes_{\mathbb{F}_p} \overline{\mathbb{F}}_p
		$$
		for all $g \in G_L$.
		\item 	The map $\Psi$ is $\operatorname{Gal}(L/K)$-equivariant when restricted to $u^{-ep/(p-1)}vl[[v]]$.
		\item If $H \in l((v)) \otimes_{\mathbb{F}_p} \overline{\mathbb{F}}_p$ is such that $\varphi(H) -H \in u^{-ep/(p-1)}vl[[v]] \otimes_{\mathbb{F}_p} \overline{\mathbb{F}}_p$, then $\Psi(\varphi(H)-H) =0$.
	\end{enumerate}
\end{cor}
\begin{proof}
	Part (1) follows immediately from Theorem~\ref{Abr}. For part (2) choose $g_0 \in \operatorname{Gal}(L/K)$ and write $g_0$ also for a lift to $G_{K_\infty}$. If $h \in vl[[v]]\otimes_{\mathbb{F}_p} \overline{\mathbb{F}}_p$ and $H \in C^\flat \otimes_{\mathbb{F}_p} \overline{\mathbb{F}}_p$ satisfies $\varphi(H)-H$ then, by part (1),
	$$
	\begin{aligned}
		\Psi(g_0(h))(g) &\equiv g(g_0(H)) -g_0(H)  	\text{ modulo } \mathfrak{m}_{C^\flat}  \\
		&= g_0\left( g_0^{-1} g g_0(H) - H \right) 	\\
		&\equiv g_0 \Psi(h)(g_0^{-1}gg_0)
		\text{ modulo } \mathfrak{m}_{C^\flat}, 
	\end{aligned}
	$$
    Since the action of $\operatorname{Gal}(L/K)$ on $H^1(G_L,\overline{\mathbb{F}}_p)$ is given by 
    $$
    (g_0 \cdot c)(g) = g_0 c(g_0^{-1}gg_0),
    $$
    part (2) follows. For part (3) we note that $\varphi(H) - H \in u^{-ep/(p-1)}vl[[v]] \otimes_{\mathbb{F}_p} \overline{\mathbb{F}}_p$ implies 
    $$
    H \in u^{-ep/(p-1)}vl[[v]] \otimes_{\mathbb{F}_p} \overline{\mathbb{F}}_p.
    $$
    Therefore, the calculation made in the proof of Theorem~\ref{Abr} shows that $g(H) - H \in \mathfrak{m}_{C^\flat}\otimes_{\mathbb{F}_p} \overline{\mathbb{F}}_p$ for all $g \in G_L$.
\end{proof}

\section{Explicit Serre weights via Kummer theory}\label{secKummer}

In this section we define a version of $W^{\operatorname{exp}}(\overline{r})$ for reducible but not necessarily semisimple $\overline{r}$. In order to apply results from the previous section we assume $p>2$.
\begin{notation}
	We frequently use the observation that $k \otimes_{\mathbb{F}_p} \overline{\mathbb{F}}_p = \prod_{\tau:k\rightarrow \overline{\mathbb{F}}_p} \overline{\mathbb{F}}_p$, the identification being given by $a\otimes b \mapsto (\tau(a)b)_\tau$. Therefore $k[[v]] \otimes_{\mathbb{F}_p} \overline{\mathbb{F}}_p = \prod_{\tau:k\rightarrow \overline{\mathbb{F}}_p} \overline{\mathbb{F}}_p[[v]]$. Via this identification we can express an element of $y \in k[[v]] \otimes_{\mathbb{F}_p} \overline{\mathbb{F}}_p$ as a tuple $(y_\tau)_\tau$ with $y_\tau \in \overline{\mathbb{F}}_p[[v]]$.
\end{notation}
\begin{defn}\label{subspacedefn}
	Fix a Serre weight $\sigma$ and two characters $\chi_1,\chi_2:G_K \rightarrow \overline{\mathbb{F}}_p^\times$. Then we define
	$$
	\Psi_{\sigma}(\chi_1,\chi_2) := \sum_{J,x} \Psi_{\sigma,J,x} \subset H^1(G_L,\overline{\mathbb{F}}_p) 
	$$
	where:
	\begin{itemize}
		\item The sum runs over pairs $J \subset \operatorname{Hom}_{\mathbb{F}_p}(k,\overline{\mathbb{F}}_p)$ and $x = (x_\tau)_{\tau:k \rightarrow \overline{\mathbb{F}}_p}$ with $x_\tau \in [0,e-1]$ for which
		$$
		\chi_1|_{I_K} = \prod_{\tau \in J} \omega_\tau^{a_\tau + 1 +x_\tau} \prod_{\tau \not\in J} \omega_\tau^{b_\tau + x_\tau}, \qquad \chi_2|_{I_K} = \prod_{\tau \not\in J} \omega_{\tau}^{a_\tau + e - x_{\tau}} \prod_{\tau \in J} \omega_{\tau}^{b_\tau + e-1 - x_{\tau}}
		$$ 
		\item $\Psi_{\sigma,J,x}$ is the image of
		$$
		(v^{\Omega_{\tau,\sigma,J,x} - (p^f-1)x_\tau})_\tau l[[u]]\otimes_{\FFp}\overline{\mathbb{F}}_p
		$$
		under $\Psi:u^{-ep/(p-1)}l[[u]]\otimes_{\FFp}\overline{\mathbb{F}}_p \rightarrow H^1(G_L,\overline{\mathbb{F}}_p)$ for 
		$$
		\Omega_{\tau,\sigma,J,x} := \sum_{i=0}^{f-1} p^i\left( (a_{\tau \circ \varphi^i} -b_{\tau \circ \varphi^i} +1)(-1)^{\tau \circ \varphi^i \in J} + (e-1-2 x_{\tau \circ \varphi^i})\right),
		$$
		where 
		$$
		(-1)^{\tau \in J} := \begin{cases}
			1 & \text{if $\tau \not\in J$;} \\
			-1 & \text{if $\tau \in J$.}
		\end{cases}
		$$
		For this to make sense we need  $(v^{\Omega_{\tau,\sigma,J,x} - (p^f-1)x_\tau})_\tau l[[u]]\otimes_{\FFp}\overline{\mathbb{F}}_p \subset u^{-ep/(p-1)}l[[u]]\otimes_{\FFp}\overline{\mathbb{F}}_p$. This can be seen by observing that $\Omega_{\tau,\sigma,J,x}$ is minimal when $J = \operatorname{Hom}_{\mathbb{F}_p}(k,\overline{\mathbb{F}}_p)$ and each $x_\tau =e-1$, in which case $\Omega_{\tau,\sigma,J,x} \geq -(e+p-1)(p^f-1)/(p-1)$. Notice also that $\Psi_{\sigma}(\chi_1,\chi_2)$ is empty unless $\sigma  \in W^{\operatorname{exp}}(\chi_1 \oplus \chi_2)$.
	\end{itemize}
\end{defn} 
\begin{rem}\label{omegawhy}
	To motivate the appearance of the value $\Omega_{\tau,\sigma,J,x}$ set
	$$
	s_\tau = \begin{cases}
		b_\tau + x_\tau & \text{if $\tau \not\in J$} \\
		a_\tau + 1 + x_\tau & \text{if $\tau \in J$}
	\end{cases}, \qquad t_\tau = \begin{cases}
		a_\tau +e-x_\tau & \text{if $\tau \not\in J$} \\
		b_\tau + e-1-x_\tau & \text{if $\tau \in J$}
	\end{cases}
	$$
	and notice that $\Omega_{\tau,\sigma,J,x} = \Omega_{\tau,t-s}$ for $\Omega_{\tau,t-s}$ defined as in \eqref{omegadef}. Therefore, for $g \in \operatorname{Gal}(L/K)$, we have
	\begin{equation}\label{rem-omega}
		g (v^{\Omega_{\tau,\sigma,J,x}})_\tau = (v^{\Omega_{\tau,\sigma,J,x}}\omega_\tau^{\Omega_{\tau,\sigma,J,x}}(g))_\tau = \omega_{\sigma,J,x}(g) (v^{\Omega_{\tau,\sigma,J,x}} )_\tau,
	\end{equation}
	where
	$\omega_{\sigma,J,x} := \prod_\tau \omega_\tau^{t_\tau -s_\tau}$. Examining the first bullet point of Definition~\ref{subspacedefn} shows that $\Psi_{\sigma}(\chi_1,\chi_2)$ being non-empty implies
	$$
	\chi_2/\chi_1|_{I_K} = \omega_{\sigma,J,x}.
	$$
\end{rem}
\begin{defn}
	Suppose $\overline{r} \sim \left( \begin{smallmatrix}
		\chi_1 & c \\ 0 & \chi_2
	\end{smallmatrix}\right)$ and write $\chi = \chi_1/\chi_2$. We define $W^{\operatorname{exp}}(\overline{r})$ by asserting that $\sigma_{a,b} \in W^{\operatorname{exp}}(\overline{r})$ if and only if
	\begin{enumerate}
		\item $\sigma \in W^{\operatorname{exp}}(\overline{r}^{\operatorname{ss}})$ and
		\item under the identification $H^1(G_K,\overline{\mathbb{F}}_p(\chi)) = H^1(G_L,\overline{\mathbb{F}}_p)^{\operatorname{Gal}(L/K) = \chi^{-1}}$ induced by the inflation-restriction exact sequence we have $c \in \Psi_{\sigma}(\chi_1,\chi_2)$.
	\end{enumerate}
\end{defn}
Our main theorem (whose proof is completed in Section~\ref{finishproof}) is then as follows.

\begin{thm}\label{main}
	For $p>2$, we have $W^{\operatorname{exp}}(\overline{r}) = W^{\operatorname{cr}}(\overline{r})$.
\end{thm}

Notice that the definition of $\Psi_{\sigma}(\chi_1,\chi_2)$ is insensitive to twisting the characters $\chi_1$ and $\chi_2$ by unramified characters. The next lemma shows that, in fact, $\Psi_{\sigma}(\chi_1,\chi_2)$ accounts for all such unramified twistings simultaneously. 
\begin{lem}\label{unram}
	We have
	$$
	\Psi_{\sigma}(\chi_1,\chi_2) = \bigoplus_{\psi} \Psi_{\sigma}(\chi_1,\chi_2)^{\operatorname{Gal}(L/K) = \psi \chi_2/\chi_1}
	$$
	with the direct sum running over all unramified characters $\psi:G_K\rightarrow \overline{\mathbb{F}}_p^\times$. 
\end{lem}

Since $\operatorname{Gal}(L/K)$ has order prime to $p$, this lemma follows from the observation that $\Psi_{\sigma}(\chi_1,\chi_2) = \Psi_{\sigma}(\chi_1,\chi_2)^{I(L/K) = \chi_2 / \chi_1}$ for $I(L/K) \subset \operatorname{Gal}(L/K)$ the inertia subgroup. In almost all cases this follows from the equivariance of $\Psi$ from Corollary~\ref{Abrcor}. In one degenerate case this argument does not work. Since we explain this issue in more detail in Section~\ref{sec-refine}, we give a complete proof of Lemma~\ref{unram} in that section.

\section{Refined descriptions of \texorpdfstring{$\Psi_{\sigma}(\chi_1,\chi_2)$}{psi-sigma-chi1-chi2}}\label{sec-refine}

In this section we give more concrete descriptions of $\Psi_{\sigma}(\chi_1,\chi_2)$ by producing smaller subspaces of $u^{-ep/(p-1)}l[[v]] \otimes_{\mathbb{F}_p}\overline{\mathbb{F}}_p$ whose image under $\Psi$ computes $\Psi_{\sigma}(\chi_1,\chi_2)^{\operatorname{Gal}(L/K) = \chi^{-1}}$. In particular, these calculations will give upper bounds on the dimension of $\Psi_{\sigma}(\chi_1,\chi_2)$.

We begin by dealing with the most degenerate situation.
\begin{lem}\label{lem-degen}
	Suppose $\sigma = \sigma_{a,b}$ with $a_\tau -b_\tau =p-1$ for every $\tau$. If $J = \operatorname{Hom}_{\mathbb{F}_p}(k,\overline{\mathbb{F}}_p)$ and $x_\tau = e-1$ for each $\tau$, then 
	$$
	\Psi_{\sigma,J,x}^{\operatorname{Gal}(L/K) = \chi^{-1}} = H^1(G_K,\overline{\mathbb{F}}_p(\chi))
	$$
	if $\chi$ is an unramified twist of the cyclotomic character $\chi_{\operatorname{cyc}}$.
\end{lem}
\begin{proof}
	The assumptions on $\sigma, J$ and $x$ imply $\Omega_{\tau,\sigma,J,x} - (p^f-1)x_\tau = -ep(p^f-1)/(p-1)$. Therefore, $\Psi_{\sigma,J,x}$ equals the $\overline{\mathbb{F}}_p$-linear extension of the image of $l[[u]]$ under the surjection $\Psi_0: l[[v]] \rightarrow L^\times \otimes_{\mathbb{Z}} \mathbb{F}_p = H^1(G_L,\mu_p(L))$. If $K_{\operatorname{ur}}$ is the unramified extension of $K$ of degree $p^f-1$, then it follows that $\Psi_{\sigma,J,x}$ equals the $\overline{\mathbb{F}}_p$-linear extension of the image of $K_{\operatorname{ur}}^\times \otimes_{\mathbb{Z}} \mathbb{F}_p = H^1(G_{K_{\operatorname{ur}}},\mu_p(L))$ under the restriction map. Thus, $\Psi_{\sigma,J,x}^{\operatorname{Gal}(L/K) = \chi^{-1}} = H^1(G_K, \overline{\mathbb{F}}_p(\chi))$ if $\chi$ is an unramified twist of the cyclotomic character (and is zero otherwise).
\end{proof}
\begin{cor}\label{cor-degen}
	Suppose $\overline{r} = \psi \otimes \left( \begin{smallmatrix}
		\chi_{\operatorname{cyc}} & c \\ 0 & 1
	\end{smallmatrix}\right)$ for an unramified character $\psi$. Then $\sigma = \sigma_{a,0} \in W^{\operatorname{exp}}(\overline{r})$ if $a_\tau = p-1$ for every $\tau$.
\end{cor}
\begin{proof}
	The previous lemma implies that $\chi_{\operatorname{cyc}}^{-1}|_{I_K} = \omega_{\sigma,J,x}$ for $J = \operatorname{Hom}_{\mathbb{F}_p}(k,\overline{\mathbb{F}}_p)$ and $x_\tau =e-1$ for every $\tau$. From this we deduce that $\sigma \in W^{\operatorname{exp}}(\overline{r}^{\operatorname{ss}})$. The previous lemma also shows that $\Psi_{\sigma}(\chi_1,\chi_2)$ contains $H^1(G_K,\overline{\mathbb{F}}_p(\psi \chi_{\operatorname{cyc}}))$ and so $c \in \Psi_{\sigma}(\chi_1,\chi_2)$.
\end{proof}

For the rest of the section we assume we are not in the case just mentioned, i.e. we assume that if $J = \operatorname{Hom}_{\mathbb{F}_p}(k,\overline{\mathbb{F}}_p)$ then $x_\tau \neq e-1$ for at least one $\tau$. The essential reason for distinguishing between these two cases is because, as mentioned in Definition~\ref{subspacedefn}, the inequality
$$
\Omega_{\tau,\sigma,J,x} \geq -(e+p-1)\frac{p^f-1}{p-1}
$$
is strict except when $J = \operatorname{Hom}_{\mathbb{F}}(k,\overline{\mathbb{F}}_p)$ and $x_\tau = e-1$ for every $\tau$. We will see in the following proofs (see also the proof of Theorem~\ref{galois}) that the strictness of this inequality plays a crucial role in certain arguments.

\begin{proof}[Proof of Lemma~\ref{unram}]
We've seen it suffices to show that $\Psi_{\sigma}(\chi_1,\chi_2) = \Psi_{\sigma}(\chi_1,\chi_2)^{I(L/K) = \chi_2 / \chi_1}$ for the inertia subgroup $I(L/K) \subset \operatorname{Gal}(L/K)$. By the above we may assume 
$$
\Omega_{\tau,\sigma,J,x} > -(e+p-1)(p^f-1)/(p-1)
$$
so that $(v^{\Omega_{\tau,\sigma,J,x} - (p^f-1)x_\tau})_\tau l[[u]]\otimes_{\FFp}\overline{\mathbb{F}}_p \subset u^{-ep/(p-1)}vl[[v]] \otimes_{\mathbb{F}_p} \overline{\mathbb{F}}_p$. By Corollary~\ref{Abrcor}, $\Psi$ is $\operatorname{Gal}(L/K)$-equivariant when restricted to $u^{-ep/(p-1)}vl[[v]] \otimes_{\mathbb{F}_p} \overline{\mathbb{F}}_p$, it is enough to show that $I(L/K)$ acts on $(v^{\Omega_{\tau,\sigma,J,x} - (p^f-1)x_\tau})_\tau l[[u]]\otimes_{\FFp}\overline{\mathbb{F}}_p)$ as $\omega_{\sigma,J,x}$. This follows from \eqref{rem-omega}.
\end{proof}

\begin{prop}\label{maximalprop}
	Assume that if $J = \operatorname{Hom}_{\mathbb{F}_p}(k,\overline{\mathbb{F}}_p)$, then $x_\tau \neq e-1$ for at least one $\tau$. Also fix $\tau_0:k \rightarrow \overline{\mathbb{F}}_p$ and an unramified character $\psi:G_K \rightarrow \overline{\mathbb{F}}_p^\times$.
	
	Define $U_{J,x,\psi} \subset k[[u]] \otimes_{\mathbb{F}_p}\overline{\mathbb{F}}_p$ as the $k\otimes_{\mathbb{F}_p} \overline{\mathbb{F}}_p$-subspace generated by those $y =(y_\tau)_\tau \in k[[u]] \otimes_{\mathbb{F}_p} \overline{\mathbb{F}}_p$ for which 
	\begin{itemize}
		\item $y_\tau \in \overline{\mathbb{F}}_p[u]$ has non-zero terms concentrated in degrees $[0,x_\tau]$ if $\tau \circ \varphi^{-1} \in J$, and
		\item non-zero terms concentrated in degree $[0,x_\tau-1]$ if $\tau \circ \varphi^{-1} \not\in J$.
		\item If $\psi = 1$ and $-\Omega_{\tau,\sigma,J,x} \in (p^f-1)\mathbb{Z}_{\geq 0}$, then $y_{\tau_0}$ may additionally have have a non-zero term in degree $u^{-\Omega_{\tau_0,\sigma,J,x}/(p^f-1)}$.
	\end{itemize}
	Then
	$$
	\Psi_{\sigma,J,x}^{\operatorname{Gal}(L/K) = \psi \omega_{\sigma,J,x}} = \Psi((v^{\Omega_{\tau,\sigma,J,x}-(p^f-1)x_\tau})_\tau \lambda_\psi U_{J,x,\psi})
	$$
	for any generator $\lambda_\psi$ of $(l\otimes_{\mathbb{F}_p} \overline{\mathbb{F}}_p)^{\operatorname{Gal}(L/K) = \psi}$.
\end{prop}

Notice that to make sense of $\lambda_\psi$ we use that $l \otimes_{k} \overline{\mathbb{F}}_p$ is the regular $\overline{\mathbb{F}}_p$-representation of $\operatorname{Gal}(l/k)$. This implies $(l\otimes_{\mathbb{F}_p} \overline{\mathbb{F}}_p)^{\operatorname{Gal}(L/K) = \psi}$ is one dimensional over $\overline{\mathbb{F}}_p$ for any unramified character $\psi$. Thus, the $\lambda_\psi$ above above exists and is uniquely determined up to scaling.

\begin{proof}
	Recall that $\omega_{\sigma,J,x}$ is the character via which $\operatorname{Gal}(L/K)$ acts on $(v^{\Omega_{\tau,\sigma,J,x}})$ (see the proof of Lemma~\ref{unram}). Therefore, any element of $\Psi_{\sigma,J,x}^{\operatorname{Gal}(L/K) = \psi \omega_{\sigma,J,x}}$  can be written as $$
	Y: = \Psi((v^{\Omega_{\tau,\sigma,J,x} -(p^f-1)x_\tau})_\tau \lambda_\psi y)
	$$
	for some $y \in  k[[u]] \otimes_{\mathbb{F}_p} \overline{\mathbb{F}}_p$. We have to show that $Y =\Psi((v^{\Omega_{\tau,\sigma,J,x}})_\tau \lambda_\psi z)$ for some $z \in (u^{-x_\tau})_\tau U_{J,x,\psi}$. The construction of $z$ will be based upon the following recursion. Define
	$$
	y_0 = (u^{-x_\tau})_\tau y, \qquad y_i = \varphi(y_{i-1})(u^{\alpha_\tau})_\tau \mu_\psi
	$$
	where $\mu_\psi = \varphi(\lambda_\psi)\lambda_\psi^{-1} \in (k \otimes_{\mathbb{F}_p} \overline{\mathbb{F}}_p)^\times$ and 
	$
	\alpha_\tau = (p\Omega_{\tau \circ \varphi,\sigma, J,x} - \Omega_{\tau,\sigma,J,x})/(p^f-1).
	$

	By linearity, we can assume that $y = e_\kappa y'$ where $y' \in \overline{\mathbb{F}}[[u]]$, $\kappa:k \rightarrow \overline{\mathbb{F}}_p$ is some embedding, and $e_\kappa \in k \otimes_{\mathbb{F}_p} \overline{\mathbb{F}}_p \cong \prod_{\tau} \overline{\mathbb{F}}_p$ is the $\kappa$-th idempotent. Since $\varphi(e_{\kappa}) = e_{\kappa \circ \varphi^{-1}}$, we can then write $y_{n} = e_{\kappa \circ \varphi^{-n}} y'_n$ with $y'_n \in \overline{\mathbb{F}}_p((u))$.
	
	\begin{clm}\label{clm1}
		One of the following must occur:
		\begin{itemize}
			\item There exists an $n$ with $y_n \in (u^{-x_\tau})_\tau U_{J,x,\psi}$ and $y_i \in (u^{-x_\tau})_\tau k[[u]] \otimes_{\mathbb{F}_p} \overline{\mathbb{F}}_p$ for all $i\leq n$;
			\item $y_n \in (u^{-x_\tau})_\tau k[[u]] \otimes_{\mathbb{F}_p} \overline{\mathbb{F}}_p$ for all $n$ and $y_n \rightarrow 0$ as $n\rightarrow \infty$;
			\item $y_n \in (u^{-x_\tau})_\tau k[[u]] \otimes_{\mathbb{F}_p} \overline{\mathbb{F}}_p$ for all $n$ and $y_f = \xi y_0$ for some $\xi \in  \overline{\mathbb{F}}_p^\times \setminus \lbrace 1 \rbrace$.
		\end{itemize}
	\end{clm}
	\begin{proof}[Proof of claim]
		Since $y_{n-1} = e_{\kappa \circ \varphi^{-n+1}} y_{n-1}'$ we have
		$$
		y_n = e_{\kappa \circ \varphi^{-n}} u^{\alpha_{\kappa \circ \varphi^{-n}}} y_{n-1}'(u^p)\mu_\psi
		$$
		where $y'_{n-1}(u^p)$ denotes the power series obtained from $y'_{n-1}$ by substituting $u$ with $u^p$.	A quick calculation also shows that  
		$$	
		\alpha_\tau = (-1)^{\tau \in J}(a_\tau -b_\tau+1) + (e-1 -2x_\tau)
		$$ 
		for all $\tau$. Now suppose $y_{n-1} \in (u^{-x_\tau})_\tau k[[u]] \otimes_{\mathbb{F}_p} \overline{\mathbb{F}}_p$ but not in $(u^{-x_\tau})_\tau U_{J,,x,\psi}$. Then either $y_{n-1}' \in u\overline{\mathbb{F}}_p[[u]]$ or $y_{n-1}' \in \overline{\mathbb{F}}_p$ and $\kappa \circ \varphi^{-n} \not\in J$. In either of these cases one has $y_n \in (u^{-x_\tau})_\tau k[[u]] \otimes_{\mathbb{F}_p} \overline{\mathbb{F}}_p$. Therefore, either the first case holds or $y_n \in (u^{-x_\tau})_\tau k[[u]] \otimes_{\mathbb{F}_p} \overline{\mathbb{F}}_p$ for all $n$.
		
		From  now on we assume $y_n \not\in (u^{-x_\tau})_\tau U_{J,x,\psi}$ for any $n$. Replacing the sequence $(y_n)_n$ by a shift $(y_{n+i})_n$ with $i$ chosen so that $\kappa \circ \varphi^{-i} =\tau_0$ for $\tau_0$ the embedding fixed at the beginning of the section allows us to assume also that $\kappa = \tau_0$. Note that this new sequence still has every term contained in $(u^{-x_\tau})_\tau k[[u]] \otimes_{\mathbb{F}_p} \overline{\mathbb{F}}_p$. 
		
		Let $N_i$ denote the $u$-adic valuation of $y_{if} \in \overline{\mathbb{F}}_p((u))$. Then $y_{if + j}$ has $u$-adic valuation
		$$
		p^j N_i + p^{j-1}\alpha_{\tau_0 \circ \varphi^{-1}} + p^{j-2} \alpha_{\tau_0 \circ \varphi^{-2}} + \ldots + p \alpha_{\tau_0 \circ \varphi^{-(j-1)}} + \alpha_{\tau_0 \circ \varphi^{-j}}.
		$$
		In particular, we see that
		$$
		N_i := p^f N_{i-1} + \Omega_{\tau_0,\sigma,J,x}.
		$$
		If the sequence of integers $N_i$ is strictly increasing then we must have that $y_n \rightarrow 0$. Therefore we assume the sequence $N_i$ is not strictly increasing.
		
		Notice that $N_{i+1}= p^fN_i + \Omega_{\tau_0,\sigma,J,x} < N_i$ if and only if $(p^f-1)N_i < -\Omega_{\tau_0,\sigma,J,x}$. In particular, $N_{i+1}<N_i$ implies $(p^f-1)N_{i+1} = (p^f-1)p^f N_i + (p^f-1)\Omega_{\tau_0,\sigma,J,x}< -\Omega_{\tau_0,\sigma,J,x}$ and so $N_{i+2} < N_{i+1}$. Since we know $y_n \in (u^{-x_\tau})_\tau k[[u]] \otimes_{\mathbb{F}_p} \overline{\mathbb{F}}_p$ for all $n$ we know the $N_i$ are bounded from below. Therefore, the sequence $N_i$ must be constant and so $-\Omega_{\tau_0,\sigma,J,x} = (p^f-1)N_0$ and
		$$
		y_f = y_0 \left( \mu_\psi \varphi(\mu_\psi) \ldots \varphi^{f-1}(\mu_\psi) \right).
		$$
		Notice that $\xi:= \ldots \varphi^{f-1}(\mu_\psi)$ is $\varphi$-invariant and so contained in $\overline{\mathbb{F}}_p^\times$. To finish the proof we have to show $\xi \neq 1$. For this note that the identity $\varphi(\lambda_\psi)\lambda_\psi^{-1} = \mu_\psi$ implies $\varphi^f(\lambda_\psi) = \xi \lambda_\psi$. Therefore $\xi =1$ implies $\lambda_\psi \in k\otimes_{\mathbb{F}_p} \overline{\mathbb{F}}_p$. Since $\operatorname{Gal}(L/K)$ acts on $\lambda_\psi$ via $\psi$ it would follow that $\psi =1$. However, if $\psi =1$ then, because $-\Omega_{\tau_0,\sigma J,x} = (p^f-1)N_0$ with $N_0 \geq 0$, we would have $y_0 \in (u^{-x_\tau})_\tau U_{J,x,\psi}$ contrary to our previous assumption. This finishes the proof of the claim. 
	\end{proof}
	
	To finish the proof, recall the observation made before the statement of the claim: that our assumptions on $J$ and $x_\tau$ ensure that $\Omega_{\tau,\sigma,J,x} >(-e+p-1)(p^f-1)/(p-1)$ for all $\tau$. Therefore
	$$
	\Omega_{\tau,\sigma,J,x} -(p^f-1)x_\tau > -ep(p^f-1)/(p-1)
	$$ 
	for all $\tau$ and so
	\begin{equation}\label{imply}
		y_n \in (u^{-x_\tau})_\tau k[[u]] \otimes_{\FFp} \overline{\mathbb{F}}_p \Rightarrow (v^{\Omega_{\tau,\sigma,J,x}})_\tau y_n \in u^{-ep/(p-1)}vk[[v]]\otimes_{\FFp}\overline{\mathbb{F}}_p.
	\end{equation}
	Now suppose the sequence $(y_i)_i$ is as in the first bullet point of the claim. Set $y^{(n)} = (v^{\Omega_{\tau,\sigma,J,x}})_\tau \lambda_\psi \sum_{i=0}^{n-1} y_i$ and observe that
	$$
	\begin{aligned}
		\varphi(y^{(n)}) - y^{(n)} &= (v^{p\Omega_{\tau \circ \varphi,J,x}})_\tau \varphi(\lambda_\psi) \sum_{i=0}^{n-1} \varphi(y_i) - (v^{\Omega_{\tau,\sigma,J,x}})_\tau \lambda_\psi \sum_{i=0}^{n-1}y_i \\
		&= (v^{\Omega_{\tau,\sigma,J,x}})_\tau \lambda_\psi \left( \sum_{i=0}^{n-1} y_{i+1} - \sum_{i=0}^{n-1} y_i \right) \\
		&= (v^{\Omega_{\tau,\sigma,J,x}})_\tau \lambda_\psi (y_{n} - y_0).
	\end{aligned}
	$$
	By \eqref{imply} we have $\varphi(y^{(n)}) - y^{(n)} \in u^{-ep/(p-1)}vl[[v]] \otimes_{\mathbb{F}_p} \overline{\mathbb{F}}_p$ and so Corollary~\ref{Abrcor} shows we can take $z = y_n$. If the sequence $(y_i)_i$ is as in the second bullet point, then the $y^{(i)}$ converge to $y^{(\infty)}$ and we have 
	$$
	(v^{\Omega_{\tau,\sigma,J,x} - (p^f-1)x_\tau})_\tau \lambda_\psi y + \varphi(y^{(\infty)}) - y^{(\infty)}= 0.
	$$
	Since  $(v^{\Omega_{\tau,\sigma,J,x} - (p^f-1)x_\tau})_\tau \lambda_\psi y \in u^{-ep/(p-1)}vl[[v]] \otimes_{\mathbb{F}_p} \overline{\mathbb{F}}_p$, the same is true of $\varphi(y^{(\infty)}) - y^{(\infty)}$. Therefore, Corollary~\ref{Abrcor} shows we can take $z=0$. Finally, if $y_f = \xi y_0$ is as in the third bullet point then, since $\xi \in \overline{\mathbb{F}}_p^\times \setminus \{1\}$, 
	$$
	(v^{\Omega_{\tau,\sigma,J,x} - (p^f-1)x_\tau})_\tau \lambda_\psi y = \varphi(\frac{y^{(f)}}{\xi-1}) - \frac{y^{(f)}}{\xi-1}.
	$$
	Again, $\varphi(\frac{y^{(f)}}{\xi-1}) - \frac{y^{(f)}}{\xi-1} \in  u^{-ep/(p-1)}vl[[v]] \otimes_{\mathbb{F}_p} \overline{\mathbb{F}}_p$, so we can again take $z = 0$.
\end{proof}

\begin{prop}\label{proposition-maximal}
	For pairs $(J,x)$ and $(J',x')$ with $J,J' \subset \operatorname{Hom}_{\mathbb{F}_p}(k,\overline{\mathbb{F}}_p)$ and $x_\tau,x_\tau' \in [0,e-1]$, write
	$$
	(J,x) \leq (J',x') \Leftrightarrow \Omega_{\tau,\sigma,J,x} - \Omega_{\tau,\sigma,J',x'} \in 2(p^f-1)\mathbb{Z}_{\geq 0} \text{ for all $\tau$.}
	$$
	Then
	\begin{enumerate}
		\item $(J,x) \leq (J',x')$ implies $\Psi_{\sigma,J,x} \subset \Psi_{\sigma,J',x'}$.
		\item Fix a Serre weight $\sigma =\sigma_{a,b}$ and $\overline{r} \sim \left( \begin{smallmatrix}
			\chi_1 & * \\ 0 & \chi_2 
		\end{smallmatrix}\right)$. Then the set of pairs $(J,x)$ for which
		$$
		\overline{r}^{\operatorname{ss}}|_{I_K} = \begin{pmatrix} 
			\prod_{\tau \in J} \omega_{\tau}^{a_\tau + 1 + x_{\tau}} \prod_{\tau \not\in J} \omega_{\tau}^{b_\tau + x_{\tau}} & 0 \\ 
			0 & \prod_{\tau \not\in J} \omega_{\tau}^{a_\tau + e - x_{\tau}} \prod_{\tau \in J} \omega_{\tau}^{b_\tau + e-1 - x_{\tau}}
		\end{pmatrix}
		$$
		contains a unique maximal element $(J_{\operatorname{max}},x_{\operatorname{max}})$.
	\end{enumerate}
\end{prop}
\begin{proof}
	Define $s_\tau = x_\tau$ for $\tau \not\in J$ and $s_\tau = a_\tau -b_\tau + 1 + x_\tau$ for $\tau \in J$. Similarly, we make sense of $s'_\tau$. Observe that $\Omega_{\tau,\sigma,J,x} - \Omega_{\tau,\sigma,J',x'} = 2 \Lambda_\tau$, where 
	$$
	\Lambda_\tau = \sum_{i=0}^{f-1} p^i (s'_{\tau\circ \varphi^i} - s_{\tau \circ \varphi^i}).
	$$
	Thus, $(J,x) \leq (J',x')$ if and only if $\Lambda_\tau \in(p^f-1)\mathbb{Z}_{\geq 0}$ for all $\tau$. We also have $\Lambda_\tau + (p^f-1)(s_\tau' - s_\tau) = p\Lambda_{\tau \circ \varphi}$ and so
	\begin{equation}\label{value+ve}
		\Omega_{\tau,\sigma,J,x} - \Omega_{\tau,\sigma,J',x'} -(p^f-1)(x_\tau -x_\tau') = \Lambda_{\tau} + p \Lambda_{\tau \circ \varphi} + (p^f-1)(s_\tau - x_\tau + x'_\tau - s_\tau').	
	\end{equation}
	If $\Lambda_{\tau\circ \varphi} > 0$, then \eqref{value+ve} is $\geq 0$, since $x_\tau' - s_\tau' \in [-p,0]$ and $s_\tau - x_\tau \ge 0$. If $\Lambda_{\tau\circ \varphi} =0$, then $\Lambda_\tau = (p^f-1)(s_\tau - s_\tau')$ and so $s_\tau \geq s'_\tau$. Therefore, $\tau \in J'$ implies $\tau \in J$ and so
	$$
	(s_\tau - x_\tau + x'_\tau - s_\tau') = \begin{cases}
		0 & \text{if $\tau \in J'$;} \\
		0 & \text{if $\tau \not\in J,J'$;} \\
		a_\tau - b_\tau + 1 &\text{if $\tau \not\in J',\tau \in J$.} 
	\end{cases}
	$$
	We conclude again that \eqref{value+ve} is $\geq 0$. Since $(J,x) \leq (J',x')$, it follows that in every case 
	$$
	(v^{\Omega_{\tau,\sigma,J,x}-(p^f-1)x_\tau})_\tau l[[u]] \otimes_{\mathbb{F}_p} \overline{\mathbb{F}}_p \subset (v^{\Omega_{\tau,\sigma,J',x'}-(p^f-1)x'_\tau})_\tau l[[u]] \otimes_{\mathbb{F}_p} \overline{\mathbb{F}}_p.
	$$
	This shows $\Psi_{\sigma,J,x} \subset \Psi_{\sigma,J',x'}$.
	
	Part (2) has already been proved in \cite[5.3.3]{gls15}, but in a different setting. To translate their statement into ours, suppose $(J,x)$ is such that 
	$$
	\overline{r}^{\operatorname{ss}}|_{I_K} \cong \begin{pmatrix} 
		\prod_{\tau \in J} \omega_{\tau}^{a_\tau + 1 + x_{\tau}} \prod_{\tau \not\in J} \omega_{\tau}^{b_\tau + x_{\tau}} & 0 \\ 
		0 & \prod_{\tau \not\in J} \omega_{\tau}^{a_\tau + e - x_{\tau}} \prod_{\tau \in J} \omega_{\tau}^{b_\tau + e-1 - x_{\tau}}
	\end{pmatrix}
	$$
	(if no such $(J,x)$ exists then there is nothing to prove). Let $s_\tau$ be defined as in the first sentence of the proof. To accommodate the notation in loc. cit. fix an embedding $\tau_0$ and set $s_i = s_{\tau_0 \circ \varphi^i}$. Notice that $s_i \in [0,e-1] \cup [r_i, r_i +e-1]$ for $r_i := a_{\tau_0 \circ \varphi^i} - b_{\tau_0\circ \varphi^i} + 1$, and so we can apply \cite[5.3.3]{gls15} with $\mathfrak{N} = \mathfrak{M}(s_0,\ldots,s_{f-1};1)$. Combining \cite[3.1.1, 3.1.2 and 5.1.2]{gls15} shows that \cite[5.3.3]{gls15} applied in this way produces $s^{\operatorname{min}}_0,\ldots,s^{\operatorname{min}}_{f-1}$ so that, for any $j \in [0,f-1]$,
	\begin{equation}\label{smax}
		\sum_{i=0}^{f-1} p^i(s^{\operatorname{min}}_{i+j}- s_{i+j}) \in (p^f-1)\mathbb{Z}_{\geq 0}	
	\end{equation}
	(here the indices of $s_i$ and $s_i^{\operatorname{min}}$ are viewed modulo $f$). It follows from the proposition that $s_i^{\operatorname{min}} \in [0,e-1] \cup [r_i, r_i +e-1]$ for each $i$. Define $J_{\operatorname{max}} \subset \operatorname{Hom}_{\mathbb{F}_p}(k,\overline{\mathbb{F}}_p)$ by asserting that $\tau_0 \circ \varphi^i \in J_{\operatorname{max}}$ if and only if $s_i^{\operatorname{min}} \not\in [0,e-1]$, and define $x_{\operatorname{max},\tau_0 \circ \varphi^i} := s_i^{\operatorname{min}}$ if $\tau_0 \circ \varphi^i \not\in J_{\operatorname{max}}$  and $x_{\operatorname{max},\tau_0 \circ \varphi^i} := s_i^{\operatorname{min}} - r_i$ otherwise. Then \eqref{smax} implies that
	$$
	\Omega_{\tau,\sigma,J,x} - \Omega_{\tau,\sigma,J_{\operatorname{max},x_{\operatorname{max}}}} \in (p^f-1)\mathbb{Z}_{\geq 0}
	$$
	for all $\tau$. In other words $(J,x) \leq (J_{\operatorname{max}},x_{\operatorname{max}})$. Since the $s_i^{\operatorname{min}}$ are independent of the chosen pair $(J,x)$ it follows that $(J_{\operatorname{max}},x_{\operatorname{max}})$ is the desired maximal pair.
\end{proof}
\begin{rem}
	The proof of \cite[5.3.3]{gls15} (or more precisely, the equivalent statement of \cite[5.3.1]{gls15}) gives an algorithm to compute the $s_i^{\operatorname{min}}$, and therefore the maximal pair $(J_{\operatorname{max}},x_{\operatorname{max}})$, explicitly.
\end{rem}
As a consequence of the previous two propositions we immediately deduce the following corollary.
\begin{cor}\label{cor-maxpairs}
	If $(J_{\operatorname{max}},x_{\operatorname{max}})$ is the maximal pair from part (2) of Proposition~\ref{proposition-maximal} and $\chi = \chi_1/\chi_2$, then
	$$
	\Psi_{\sigma}(\chi_1,\chi_2) = \Psi_{\sigma,J_{\operatorname{max}},x_{\operatorname{max}}}
	$$
	and
	$$
	\begin{aligned}
		\operatorname{dim}_{\overline{\mathbb{F}}_p} \Psi_{\sigma}(\chi_1,\chi_2)^{\operatorname{Gal}(L/K) = \chi^{-1}} &\leq \nu + \sum_{\tau} \begin{cases}
			x_{\operatorname{max},\tau} + 1 & \text{if $\tau \circ \varphi^{-1} \in J_{\operatorname{max}}$} \\
			x_{\operatorname{max},\tau} & \text{if $\tau \circ \varphi^{-1} \not\in J_{\operatorname{max}}$}
		\end{cases} \\
		&= \nu + \operatorname{Card}(J_{\operatorname{max}})  + \sum_{\tau}x_{\operatorname{max},\tau}, 
	\end{aligned}
	$$
	where $\nu =0$ unless $\chi =1$ and $-\Omega_{\tau, J_{\operatorname{max}},x_{\operatorname{max}}} \in (p^f-1)\mathbb{Z}_{\geq 0}$ for one (equivalently all) $\tau$, in which case $\nu = 1$.
\end{cor}

\section{Breuil--Kisin modules}

We do not assume that $p>2$ in this section. Let $\mathbb{F}$ be a finite extension of $\mathbb{F}_p$, sufficiently large that there is an embedding $k \hookrightarrow \mathbb{F}$. A Breuil--Kisin module $\mathfrak{M}$ over $\mathbb{F}$ is a finite free $\mathfrak{S}_{\mathbb{F}} := k[[u]] \otimes_{\mathbb{F}_p} \mathbb{F}$-module equipped with a homomorphism
$$
\varphi: \mathfrak{M} \otimes_{\mathfrak{S}_{\mathbb{F}},\varphi} \mathfrak{S}_{\mathbb{F}} \rightarrow \mathfrak{M}
$$
with cokernel killed by a power of $E(u) \in \mathfrak{S}_{\mathbb{F}}$.\footnote{One can define Breuil--Kisin modules over more general $\mathbb{Z}_p$-algebras but in this paper we only need to consider $p$-torsion Breuil--Kisin modules over $\mathbb{F}$.} Here $\varphi$ on $\mathfrak{S}_{\mathbb{F}}$ denotes $\mathbb{F}$-linear extension of the the $p$-th power map on $k[[u]]$ and $E(u)$ denotes the (reduction modulo $p$ of the) minimal polynomial over $W(k)$ of $\pi$. Thus $E(u) = u^e$.

\begin{prop}\label{crys}
	If $r$ is a crystalline representation over $\mathcal{O}$ with Hodge--Tate weights $\geq 0$ and $r \otimes_{\mathcal{O}} \mathbb{F} \cong \overline{r}$, then there exists a Breuil--Kisin module $\mathfrak{M}$ over $\mathbb{F}$ and a continuous $\varphi$-equivariant $C^\flat$-semilinear $G_K$-action on $\mathfrak{M} \otimes_{k[[u]]} C^\flat$ such that
	\begin{enumerate}
		\item[(D1)] $\sigma(x) - x \in \mathfrak{M} \otimes_{k[[u]]} u^{(e+p-1)/(p-1)} \mathcal{O}_{C^\flat}$ for all $\sigma \in G_K$ and $x \in \mathfrak{M}$,
		\item[(D2)] $\sigma(x) = x$ for $\sigma \in G_{K_\infty}$ and $x \in \mathfrak{M}$,
	\end{enumerate}
	and such that $\varphi,G_K$-equivariantly
	$$
	\mathfrak{M} \otimes_{k[[u]]} C^\flat \cong \overline{r}^\vee \otimes_{\mathbb{F}_p} C^\flat,
	$$
	where the Frobenius on the left hand side is that fixing $\overline{r}^\vee$ (so we can identify $\overline{r}^\vee = (\mathfrak{M} \otimes_{k[[u]]} C^\flat)^{\varphi=1}$).
\end{prop}	
\begin{proof}
	This follows by applying \cite[2.1.12]{B20} to the crystalline representation $r^\vee$ (note that Hodge--Tate weights in \cite{B20} are normalised to be the negative of those here so $r^\vee$ has Hodge--Tate weights $\geq 0$ in the sense of loc. cit.) and base-changing along $\mathcal{O} \rightarrow \mathbb{F}$. Here we also use the observation that the image of the element $\mu := [\epsilon] -1$ in loc. cit. modulo~$p$ generates the ideal $u^{ep/(p-1)}\mathcal{O}_{C^\flat}$ (cf. \cite[5.1.3]{Fon94}), and so $[\pi^\flat] \varphi^{-1}(\mu)$ modulo $p$ generates the ideal $u^{1+ep/(p-1)}\mathcal{O}_{C^\flat} = u^{(e+p-1)/(p-1)}\mathcal{O}_{C^\flat}$.
\end{proof}

The following is (a minor alteration of) the key technical result in \cite{gls15}.
\begin{prop}\label{gls}
	Suppose that $r$ has Hodge type $\sigma_{a,0}$ and that $r \otimes_{\mathcal{O}} \mathbb{F}$ is reducible. Let $\mathfrak{M}$ be the mod~$p$  Breuil--Kisin module associated to $r$ as in Proposition~\ref{crys}. Then there exist integers $0 \leq s_\tau,t_\tau \leq p$ with
	\begin{equation}\label{s,tcond}
		s_\tau + t_\tau = a_\tau +e, \qquad \operatorname{max}\lbrace s_\tau,t_\tau \rbrace \geq a_\tau +1
	\end{equation}
	and an $\mathfrak{S}_{\mathbb{F}}$-basis $\beta$ of $\mathfrak{M}$ so that
	$$
	\varphi(\beta) = \beta\begin{pmatrix}
		x_1(u^{t_\tau}) & (y_\tau) \\ 0 & x_2(u^{s_\tau})
	\end{pmatrix}
	$$
	for some $x_1,x_2 \in (k \otimes_{\mathbb{F}_p} \mathbb{F})^\times$ and $y_\tau \in u^{\delta_\tau}\overline{\mathbb{F}}_p[[u]]$, where
	$$
	\delta_\tau = \begin{cases}
		a_\tau +1 & \text{if $t_\tau < a_\tau+1$;}\\
		0 & \text{if $t_\tau \geq a_\tau+1$.}
	\end{cases}
	$$
\end{prop}
\begin{proof}
	As we will explain this follows from \cite[5.1.5]{gls15}. Notice however that in loc. cit. it is assumed that $p>2$. This requirement has been removed by the work of \cite{wan17} where it is shown that \cite[5.1.5]{gls15} remains true, provided the uniformiser $\pi \in K$ is chosen as in \cite[2.1]{wan17}. 
	
	It follows from \cite[5.1.5]{gls15} there are $s_\tau,t_\tau$ satisfying \eqref{s,tcond} and a basis $\beta$ of $\mathfrak{M}$ satisfying $\varphi(\beta) = \beta \left( \begin{smallmatrix}
		x_1(u^{t_\tau}) & (y'_\tau) \\ 0 & x_2(u^{s_\tau})
	\end{smallmatrix} \right)$ for some $x_1,x_2 \in (k\otimes_{\mathbb{F}_p} \mathbb{F})^\times$ and polynomials $y'_\tau \in \overline{\mathbb{F}}_p[u]$ as claimed, except that possibly $y'_\tau$ has a non-zero term of degree $t_\tau$ if $t_\tau < a_\tau+1$. A straightforward change of basis argument allows us to remove these $u^{t_\tau}$ terms at the cost of introducing terms of degree $\geq \delta_\tau$. This gives the formulation here. 
\end{proof}

We conclude this section by explaining how one can describe the restriction of $\overline{r}^\vee = (r \otimes_{\mathcal{O}} \mathbb{F})^\vee$ to $G_{K_\infty}$ in terms of the matrices from Proposition~\ref{gls}.  We do this by setting 
\begin{equation}\label{D}
	D =  \begin{pmatrix}
		d_1 & dd_1 \\ 0 & d_2
	\end{pmatrix} \in \operatorname{Mat}(C^\flat \otimes_{\mathbb{F}_p} \mathbb{F})
\end{equation}
with entries defined by the equations
$$
d_1 = \varphi(d_1)	x_1(u^{t_\tau}), \qquad d_2 = \varphi(d_2)x_2(u^{s_\tau}), \qquad \varphi(d) -d = -\frac{d_2}{d_1x_2} ( u^{-s_\tau}y_\tau).
$$
Then $D = \left( \begin{smallmatrix}
	x_1(u^{t_\tau}) & (y_\tau) \\ 0 & x_2(u^{s_\tau})
\end{smallmatrix} \right) \varphi(D)$ and so, if $\alpha = \beta D$ in $\mathfrak{M} \otimes_{k[[u]]} C^\flat$ then $\varphi(\alpha) = \alpha$. Therefore, $\alpha$ is an $\mathbb{F}$-basis of $\overline{r}^\vee = (\mathfrak{M} \otimes_{k[[u]]} C^\flat)^{\varphi=1}$ and (D2) from Proposition~\ref{crys} implies
\begin{equation}\label{inftyaction}
	\sigma(\alpha) = \alpha D^{-1} \sigma(D) = \alpha \begin{pmatrix}
		\frac{\sigma(d_1)}{d_1} & \sigma(d)\frac{\sigma(d_1)}{d_1}  - d \frac{\sigma(d_2)}{d_2}\\ 0 & \frac{\sigma(d_2)}{d_2}
	\end{pmatrix}	
\end{equation}
for $\sigma  \in G_{K_\infty}$. The elements $d_1$ and $d_2$ can be easily be described and this allows us to compute the characters appearing on the diagonal of \eqref{inftyaction}. First note that, by an easy calculation,
%  \footnote{Here is the calculation: we have $k[[u]] \otimes_{\mathbb{F}_p} \mathbb{F} \cong \prod \mathbb{F}[[u]]$ via $a\otimes b \mapsto (\tau(a)b)_\tau$. Thus $\varphi((u^{n_\tau})) =(u^{pn_{\tau \circ \varphi}})_\tau$. Therefore
	% $$
	% \varphi((v^{\Omega_{\tau, t}})_\tau) = (v^{p\Omega_{\tau \circ \varphi, t}})_\tau = (v^{\Omega_{\tau, t}} u^{t_\tau})_\tau
	% $$
	% with the second equality following from the fact that $p\Omega_{\tau, t} = pt_{\tau \circ \varphi} + \ldots + p^f t_\tau = \Omega_{\tau, t} + (p^f-1)t_\tau$.} 
we can write $d_1 = \widetilde{x}_1 (v^{-\Omega_{\tau, t}})$ for $\widetilde{x}_1 \in (l \otimes_{\mathbb{F}_p} \mathbb{F})^\times$ satisfying $\widetilde{x}_1 = \varphi(\widetilde{x}_1) x_1$ and $\Omega_{\tau, t} = \sum_{i=0}^{f-1} p^i t_{\tau \circ \varphi^i}$ as in \eqref{omegadef}. Similarly to the calculation in (\ref{rem-omega}), we find
$$
\sigma((v^{\Omega_{\tau, t}})_\tau) (v^{-\Omega_{\tau, t}})_\tau = (\omega_\tau(\sigma)^{\Omega_{\tau, t}})_\tau = \prod_\tau \omega_\tau(\sigma)^{t_\tau} 
$$
for $\sigma \in G_{K_\infty}$. On the other hand, $\widetilde{x}_1 \in (l \otimes_{\mathbb{F}_p} \mathbb{F})^\times$ and so $G_{K_\infty}$ acts on $\widetilde{x}_1$ by multiplication with an unramified character $\psi_1$. Therefore,
$$
\sigma(d_1)d_1^{-1} = \psi_1(\sigma) \prod_\tau \omega_\tau(\sigma)^{-t_\tau}.
$$
Similarly with $d_1$ replaced by $d_2$ and $t_\tau$ with $s_\tau$. We conclude that 
\begin{equation}\label{diagonalofr}
	\overline{r}^\vee|_{G_{K_\infty}} \cong \begin{pmatrix}
		\psi_1 \prod_\tau \omega_\tau^{-t_\tau} & c' \\ 0 & \psi_2 \prod_\tau \omega_\tau^{-s_\tau}
	\end{pmatrix}
\end{equation}
for some cocycle $c'$. From this one easily deduces the following corollary; note that this is exactly the argument used to prove Theorem~\ref{semisimple} in \cite{gls15}.

\begin{cor}\label{crimpliesexpinsemisimplecase}
	If $\sigma_{a,0} \in W^{\operatorname{cr}}(\overline{r})$ with $\overline{r}$ reducible, then $\sigma_{a,0} \in W^{\operatorname{exp}}(\overline{r}^{\operatorname{ss}})$.
\end{cor}
\begin{proof}
	Choose a finite extension $E/\mathbb{Q}_p$ with integers $\mathcal{O}$ so that a crystalline $\overline{\mathbb{Z}}_p$-representation $r$ witnessing $\sigma_{a,0} \in W(\overline{r})$ is defined over $\mathcal{O}$. Take $\mathbb{F}$ equal the residue field of $\mathcal{O}$. Enlarging $E$ if necessary we can assume that there is an embedding $k \hookrightarrow \mathbb{F}$ so the above results apply. In particular, applying Theorem~\ref{crys} to $r$ produces a Breuil--Kisin module $\mathfrak{M}$ with shape as in Proposition~\ref{gls}. 
	
	From \eqref{diagonalofr} it follows that $\overline{r}|_{G_{K_\infty}} \cong \left( \begin{smallmatrix}
		\psi_2 \prod_\tau \omega_\tau^{s_\tau} & c \\ 0 & \psi_1 \prod_\tau \omega_\tau^{t_\tau}
	\end{smallmatrix}\right)$ for some cocycle $c$ and $s_\tau,t_\tau$ satisfying \eqref{s,tcond}. Since restriction induces an equivalence between semi-simple  representations of $G_K$ and semi-simple representations of $G_{K_\infty}$ (see, for example, \cite[2.2.1]{Bar22}), it follows that 
	$$
	\overline{r}^{\operatorname{ss}}|_{I_K} \cong \left( \begin{smallmatrix}
		\prod_\tau \omega_\tau^{s_\tau} & 0 \\ 0 & \prod_\tau \omega_\tau^{a_\tau+e -s_\tau} 
	\end{smallmatrix}\right)= \begin{pmatrix} 
		\prod_{\tau \in J} \omega_{\tau}^{a_\tau + 1 + x_{\tau}} \prod_{\tau \not\in J} \omega_{\tau}^{x_{\tau}} & 0 \\ 
		0 & \prod_{\tau \not\in J} \omega_{\tau}^{a_\tau + e - x_{\tau}} \prod_{\tau \in J} \omega_{\tau}^{e-1 - x_{\tau}}
	\end{pmatrix},
	$$ 
	where $J = \lbrace \tau \in \operatorname{Hom}_{\mathbb{F}_p}(k,\overline{\mathbb{F}}_p)\mid t_\tau < a_\tau+1 \rbrace$ and 
	$$
	x_\tau = \begin{cases}
		s_\tau - a_\tau -1 & \text{if $\tau \in J$;} \\
		s_\tau & \text{if $\tau \not\in J$.}
	\end{cases}
	$$
	Notice that $\tau \in J$ implies $a_\tau +e \geq s_\tau \geq a_\tau+1$ and so $x_\tau \in [0,e-1]$ and likewise if $\tau \not\in J$. It follows that $\sigma_{a,0} \in W^{\operatorname{exp}}(\overline{r})^{\operatorname{ss}}$, as required.
\end{proof}
\section{Constructing Galois actions}
The goal here is to show that the $G_{K_\infty}$-action on $\overline{r}^\vee$ described in \eqref{inftyaction} can be extended to a description of the whole $G_K$-action. We point out that Theorem~\ref{galois} does not require the assumption that $p>2$.

\begin{thm}\label{galois}
	Assume $\mathfrak{M}$ is a Breuil--Kisin module over $\mathbb{F}$ with shape as in Proposition~\ref{gls}. Assume additionally that
	$$
	(s_\tau,t_\tau) \neq (e+p-1,0)
	$$ 
	for at least one $\tau$. Then there exists a unique continuous $\varphi$-equivariant $C^\flat$-semilinear action of $G_K$ on $\mathfrak{M} \otimes_{k[[u]]} C^\flat$ satisfying:
	\begin{enumerate}
		\item[(D1)] $\sigma(x) - x \in \mathfrak{M} \otimes_{k[[u]]} u^{(e+p-1)/(p-1)} \mathcal{O}_{C^\flat}$ for all $\sigma \in G_K$ and $x \in \mathfrak{M}$;
		\item[(D2)] $\sigma(x) = x$ for $\sigma \in G_{K_\infty}$ and $x \in \mathfrak{M}$.
	\end{enumerate}
	Furthermore, if $\beta$ is a basis of $\mathfrak{M}$ with $\varphi(\beta) = \beta C$, then this $G_K$-action can be described concretely by $\sigma(\beta) = C_\sigma \beta$ for
	$$
	C_\sigma = \operatorname{lim}_{n \rightarrow \infty} \left( C\varphi(C) \ldots \varphi^{n}(C) \varphi^{n}(\sigma(C^{-1})) \ldots \varphi(\sigma(C^{-1})) \sigma(C^{-1}) \right) \in \operatorname{Mat}(C^\flat \otimes_{\mathbb{F}_p} \mathbb{F}).
	$$
\end{thm}

We point out that uniqueness of this $G_K$-action can also be deduced from \cite[6.1.3]{gls15} (though in loc. cit. the language of $\varphi,\hat{G}$-modules is used).

\begin{proof}
	As we will explain below, in most cases the theorem follows from an application of \cite[11.3]{B21}. Unfortunately, these results do not apply in the special case where $(s_\tau,t_\tau) = (0,p+e-1)$ for every $\tau$. We treat this special case directly at the end of the proof. (Note that this special case is not excluded by the assumption in the statement since the condition on $(s_\tau, t_\tau)$ is ordered in the opposite way.)
	
	For now assume $(s_\tau,t_\tau) \neq (0,p+e-1)$ for some $\tau$. Since we also have $(s_\tau,t_\tau) \neq (p+e-1,0)$ for some $\tau$, there are integers $0 \leq q_\tau \leq e+p-1$ not all equal to $e+p-1$ such that
	\begin{itemize}
		\item $(u^{q_\tau}) \mathfrak{M} \subset \mathfrak{M}^\varphi \subset \mathfrak{M}$ for $\mathfrak{M}^\varphi$ the image of the linearised Frobenius on $\mathfrak{M}$.
	\end{itemize} 
	Indeed, if $\beta$ is a basis of $\mathfrak{M}$ with $\varphi(\beta) = \beta C$, then $(u^{q_\tau})\mathfrak{M}$ is generated by $\varphi(\beta) (u^{q_\tau})C^{-1}$, and so the assertion is equivalent to asking that $(u^{q_\tau})C^{-1} \in \operatorname{Mat}(k[[u]] \otimes_{\mathbb{F}_p} \mathbb{F})$. Following \cite[11.1]{B21}, the choice of $\beta$ also allows us to define a ``naive'' $C^\flat$-semilinear $G_K$-action $\sigma_{\operatorname{naive},\beta}$ on $\mathfrak{M} \otimes_{k[[u]]} C^\flat$ by $C^\flat$-semilinearly extending the $G_K$-action which is trivial on $\varphi(\beta)$. Typically, $\sigma_{\operatorname{naive},\beta}$ will not be $\varphi$-equivariant. However, \cite[11.3]{B21} (and its proof) describes conditions on $\mathfrak{M}$ which ensure $\varphi^n \circ \sigma_{\operatorname{naive},\beta} \circ \varphi^{-n}$
	converges to a $\varphi$-equivariant $C^\flat$-semilinear $G_K$-action on $\mathfrak{M} \otimes_{k[[u]]} C^\flat$ satisfying (D1) and (D2). To explain this notice that if $\mathfrak{M}$ satisfies the previous bullet point and
	\begin{itemize}
		\item $\sigma_{\operatorname{naive},\beta}(x) - x \in \mathfrak{M} \otimes_{k[[u]]} u^{(e+p-1)/(p-1)}\mathcal{O}_{C^\flat}$ for every $x \in \mathfrak{M}$ and $\sigma \in G_K$, 
	\end{itemize}
	then $(\mathfrak{M},\beta)$ defines an object in the category denoted $\widetilde{Z}_2^{\nabla_\sigma,r}(\mathbb{F})$ in \cite[11.2]{B21}. The assertion of \cite[11.3]{B21} implies that there exists a unique $G_K$-action $\sigma$ on $\mathfrak{M} \otimes_{k[[u]]} \mathcal{O}_{C^\flat}$ making $(\mathfrak{M},\sigma)$ into an object of $Y^{\leq h}_d(\mathbb{F})$. Considering the definition of $Y^{\leq h}_2(\mathbb{F})$ from \cite[10.2]{B21} we see this is equivalent to asking that $\sigma$ satisfies $(D1)$ and $(D2)$. Finally, examining the proof of \cite[11.3]{B21} shows that this action $\sigma$ is obtained as the limit of $\varphi^n \circ \sigma_{\operatorname{naive},\beta} \circ \varphi^{-n}$ as claimed.
	
	To see the formula for $C_\sigma$ in the theorem note that, by definition, $\sigma_{\operatorname{naive},\beta}(\beta) = \beta C \sigma(C)^{-1}$. Therefore,
	$$
	\begin{aligned}
		\varphi^n \circ \sigma_{\operatorname{naive},\beta} \circ \varphi^{-n}(\beta) = \varphi^n\circ \sigma_{\operatorname{naive},\beta}\left( \beta \varphi^{-1}(C^{-1}) \varphi^{-2}(C^{-1}) \ldots \varphi^{-n}(C^{-1})  \right)  \\
		= \varphi^n\left(\beta C\sigma(C^{-1}) \varphi^{-1}(\sigma(C^{-1})) \varphi^{-2}(\sigma(C^{-1})) \ldots \varphi^{-n}(\sigma( C^{-1}))\right) \\
		= \beta \left( C \varphi(C) \ldots \varphi^{n-1}(C) \varphi^n(C) \varphi^n(\sigma(C^{-1})) \varphi^{n-1}(\sigma(C^{-1})) \ldots \varphi(\sigma(C^{-1})) \sigma(C^{-1})\right). 
	\end{aligned}
	$$
	We've already seen the first bullet point holds. To apply these results we need to check the second does also.
	
	Concretely, since $\sigma_{\operatorname{naive},\beta}(\beta) = \beta C\sigma(C^{-1})$, the second bullet point is asserting that $C\sigma(C^{-1})-1 \in u^{(e+p-1)/(p-1)}\operatorname{Mat}(\mathcal{O}_{C^\flat} \otimes_{\mathbb{F}_p} \mathbb{F})$. To check this we take $\beta$ a basis as in Proposition~\ref{gls} so that
	$$
	\begin{aligned}
		C \sigma(C^{-1}) -1 &= \begin{pmatrix}
			x_1( u^{t_\tau}) & (y_\tau) \\ 0 & x_2(u^{s_\tau}) \end{pmatrix}\begin{pmatrix}
			x_1^{-1}(\sigma (u)^{-t_\tau}) & -\frac{1}{x_1x_2}(\sigma(y_\tau) \sigma(u)^{-e-a_\tau}) \\ 0 & x_2^{-1}(\sigma(u)^{-s_\tau})
		\end{pmatrix} -1 \\ &= \begin{pmatrix}
			\left(\left(\frac{u}{\sigma(u)}\right)^{t_\tau} - 1\right) & \frac{1}{x_2\sigma(u)^{s_\tau}}\left( y_\tau - \sigma(y_\tau)\left(\frac{u}{\sigma(u)}\right)^{t_\tau} \right) \\ 0 & \left(\left(\frac{u}{\sigma(u)}\right)^{s_\tau} - 1\right)
		\end{pmatrix}.
	\end{aligned}
	$$
	For the required divisibility we use that $\frac{u}{\sigma(u)}-1 \in u^{ep/(p-1)}\mathcal{O}_{C^\flat}$ (see the argument given in the proof of Theorem~\ref{Abr}). This clearly implies the required divisibility for the diagonal entries. For the upper right entry, note that
	$$
	y_\tau - \sigma(y_\tau)\left(\frac{u}{\sigma(u)}\right)^{t_\tau} = y_\tau - \sigma(y_\tau) +\sigma(y_\tau)\left(1 - \left( \frac{u}{\sigma(u)}\right)^{t_\tau}\right)
	$$ 
	is divisible by $u^{\delta_\tau + ep/(p-1)}$ because $y_\tau$ is divisible by $(u^{\delta_\tau})_\tau$ (here $\delta_\tau$ is the integer defined in \ref{gls}). Therefore we just need that $\delta_\tau -s_\tau + ep/(p-1) \geq (e+p-1)/(p-1)$. This inequality follows from the observation that $(e+p-1)/(p-1) - ep/(p-1) = -e+1$ and
	$$
	\delta_\tau - s_\tau = \begin{cases}
		a_\tau+1-s_\tau & \text{if $t_\tau <a_\tau+1$, in which case $a_\tau+1 - s_\tau = t_\tau -e+1 \geq -e+1$;} \\
		- s_\tau & \text{if $t_\tau \geq a_\tau+1$, in which case $-s_\tau = t_\tau -a_\tau -e \geq -e+1$.}
	\end{cases}
	$$
	This proves the theorem under the assumption that $(s_\tau,t_\tau) \neq (0,e+p-1)$ for at least one $\tau$.
	
	We conclude by addressing the case where $(s_\tau,t_\tau) = (0,e+p-1)$ for every $\tau$. This case is particularly simple because we can choose a basis $\beta$ of $\mathfrak{M}$ so that $\varphi(\beta) = \beta C$ with $C = \left( \begin{smallmatrix}
		x_1 u^{e+p-1} & 0 \\ 0 & x_2
	\end{smallmatrix}\right)$. Indeed, a priori, we have $C= \left( \begin{smallmatrix}
		x_1 u^{e+p-1} & y \\ 0 & x_2
	\end{smallmatrix}\right)$. Via a change of basis, we can replace $C$ by 
	$$\begin{pmatrix}
		1 & -x \\ 0 & 1
	\end{pmatrix} \begin{pmatrix}
		x_1 u^{e+p-1} & y \\ 0 & x_2
	\end{pmatrix}\begin{pmatrix}
		1 & \varphi(x) \\ 0 & 1
	\end{pmatrix} = \begin{pmatrix}
		x_1 u^{e+p-1} & y - x x_2 + \varphi(x)x_1 u^{e+p-1} \\ 0 & x_2
	\end{pmatrix}
	$$
	for any $x \in k[[u]] \otimes_{\mathbb{F}_p} \mathbb{F}$. If $y_0 = yx_2^{-1}$ and $y_{i+1}x_2 = \varphi(y_i)x_1u^{e+p-1}$, then $x= \sum_{i \geq 0} y_i$ converges in $k[[u]] \otimes_{\mathbb{F}_p} \mathbb{F}$ and satisfies $y - xx_2 +\varphi(x)x_1u^{e+p-1}$. Therefore, we can assume $C$ is diagonal.
	
	Next we show that $\mathfrak{M} \otimes_{k[[u]]} C^\flat$ can admit at most one $G_K$-action as in the theorem. Using condition (D2) and the calculations from the proof of Proposition~\ref{gls} we deduce that $\overline{r}^\vee = (\mathfrak{M} \otimes_{k[[u]]} C^\flat)^{\varphi=1}$ is generated by $\alpha = \beta D$ for $D = \left( \begin{smallmatrix}
		\widetilde{x}_1v^{-\Omega} & 0 \\ 0 & \widetilde{x}_2
	\end{smallmatrix}\right)$ with $\Omega = (p+e-1)(1+\ldots+p^{f-1})$ and $\widetilde{x}_i= \varphi(\widetilde{x}_i)x_i$. Furthermore, for $\sigma \in G_{K_\infty}$, we have $\sigma(\alpha) = \alpha \left( \begin{smallmatrix}
		\psi_1(\sigma) \prod_\tau \omega_{\tau}(\sigma)^{-(e+p-1)} & 0 \\ 0 & \psi_2(\sigma)
	\end{smallmatrix}\right)$ with $\psi_i$ defined by $\sigma(\widetilde{x}_i)  = \widetilde{x}_i \psi_i(\sigma)$. It follows that the $G_K$-action on $\alpha$ must be of the form $\sigma(\alpha) = \alpha \left( \begin{smallmatrix}
		\psi_1(\sigma)\prod_\tau \omega_{\tau}(\sigma)^{-(e+p-1)} & c(\sigma) \\ 0 & \psi_2(\sigma)
	\end{smallmatrix}\right)$ for $c:G_K \rightarrow \mathbb{F}$ a cocycle vanishing on $G_{K_\infty}$. (In fact, such non-zero cocycles occur only in very specific situations by \cite[5.4.2]{gls15}.) But then
	\begin{align*}
	\sigma(\beta) &= \beta \begin{pmatrix}
		\widetilde{x}_1 v^{-\Omega} & 0 \\ 0 & \widetilde{x}_2
	\end{pmatrix} \begin{pmatrix}
		\psi_1(\sigma) \prod_\tau \omega_{\tau}(\sigma)^{-(e+p-1)} & c \\ 0 & \psi_2(\sigma)
	\end{pmatrix} \begin{pmatrix}
		\sigma(\widetilde{x}_1)^{-1}\sigma(v)^{\Omega} & 0 \\ 0 & \sigma(\widetilde{x}_2)^{-1}
	\end{pmatrix} \\
	&= \beta \begin{pmatrix}
		\left( \frac{\sigma(v)}{v}\right)^{\Omega} \prod_\tau \omega_{\tau}(\sigma)^{-(e+p-1)} & \widetilde{x}_1 \sigma(\widetilde{x}_2)^{-1}v^{-\Omega} c \\0 & 1
	\end{pmatrix}.
	\end{align*}
	Clearly, this $G_K$-action only satisfies condition (D1) if $c=0$, which proves uniqueness. To finish the proof we just have to show that the limit formula in the theorem converges to a $G_K$-action as claimed -- this $G_K$-action will then coincide with that given in the previous formula when $c=0$. But this is clear because
	$$
	C\varphi(C) \ldots \varphi^{n}(C) \varphi^{n}(\sigma(C^{-1})) \ldots \varphi(\sigma(C^{-1})) \sigma(C^{-1}) - 1 = \begin{pmatrix}
		\left( \frac{u }{\sigma(u)} \right)^{(e+p-1)(1+\ldots+p^n)} -1 & 0 \\ 0 & 0
	\end{pmatrix}
	$$
	converges and the limit is an element of $u^{(e+p-1)/(p-1)}\operatorname{Mat}(\mathcal{O}_{C^\flat} \otimes_{\mathbb{F}_p} \mathbb{F})$ as $n\rightarrow \infty$, since $\left( \frac{u}{\sigma(u)} \right) -1 \in u^{(e+p-1)/(p-1)}\mathcal{O}_{C^\flat}$.
\end{proof}
\begin{cor}\label{ASfinal}
	Suppose $p>2$ and $\sigma_{a,0} \in W^{\operatorname{cr}}(\overline{r})$ for $\overline{r} \cong \left( \begin{smallmatrix}
		\chi_1 & c \\ 0 & \chi_2
	\end{smallmatrix}\right)$. If $a_\tau = p-1$ for each $\tau$, then assume $\chi= \chi_1/\chi_2$ is not an unramified twist of the cyclotomic character. Then $\sigma_{a,0} \in W^{\operatorname{exp}}(\overline{r})$.
\end{cor}
\begin{proof}
	As in Corollary~\ref{crimpliesexpinsemisimplecase} choose a finite extension $E/\mathbb{Q}_p$, with integers $\mathcal{O}$ and residue field $\mathbb{F}$, so that a crystalline $\overline{\mathbb{Z}}_p$-representation $r$ witnessing $\sigma_{a,0} \in W(\overline{r})$ is defined over $\mathcal{O}$. After possibly enlarging $E$ we can apply Theorem~\ref{crys} to $r$ to produce a Breuil--Kisin module $\mathfrak{M}$ with shape as in Proposition~\ref{gls}. 
	
	We want to apply Theorem~\ref{galois} to $\mathfrak{M}$. Let $s_\tau, t_\tau, y_\tau, \delta_\tau$ and $x_i$ be as in Proposition~\ref{gls}. We need to show that 
	$$
	(s_\tau,t_\tau) \neq (e+p-1,0)
	$$
	for at least one $\tau$. If not, then the identity $s_\tau + t_\tau = a_\tau+e$ would imply $a_\tau =p-1$ for every $\tau$. From \eqref{diagonalofr} it would also follow that $\chi|_{I_K} = \prod_\tau \omega_\tau^{e+p-1}$. However, $\chi_{\operatorname{cyc}}|_{I_K} = \prod_\tau \omega_\tau^{e+p-1}$, so $\chi$ would be an unramified twist of the cyclotomic character. Since we've assumed this isn't the case we can apply Theorem~\ref{galois} to produce a $G_K$-action on $\mathfrak{M} \otimes_{k[[u]]} C^\flat$. By uniqueness, this $G_K$-action coincides with that induced from the $G_K$-action on $\overline{r}^\vee$.
	
	In the discussion after the proof of Proposition~\ref{gls} we described a basis $\alpha = \beta D$ of $\overline{r}^\vee = (\mathfrak{M}~\otimes_{k[[u]]}~C^\flat)^{\varphi=1}$ with $D = \left( \begin{smallmatrix}
		d_1 & dd_1 \\ 0 & d_2
	\end{smallmatrix} \right)$ for $d_1 = \widetilde{x}_1(v^{-\Omega_{\tau,t}})_\tau$, $d_2 = \widetilde{x}_2(v^{-\Omega_{\tau, s}})_\tau$ and
	$$
	\varphi(d) - d = -\frac{d_2}{d_1x_2} ( u^{-s_\tau}y_\tau)_\tau.
	$$
	Since $x_2 \in (k\otimes_{\mathbb{F}_p} \mathbb{F})^\times$ and $\widetilde{x}_i \in (l\otimes_{\mathbb{F}_p} \mathbb{F})^\times$, it follows that $\varphi(d)-d \in (v^{\Omega_{\tau, t-s} + (p^f-1)(\delta_\tau- s_\tau)})_\tau l[[u]] \otimes_{\mathbb{F}_p} \mathbb{F}$. Using the description of the $G_K$-action from Theorem~\ref{galois} we will compute $\sigma(\alpha)$ for $\sigma \in G_K$. Since $\alpha = \varphi(\alpha) = \varphi(\beta)\varphi(D)$, we have $\sigma_{\operatorname{naive},\beta}(\alpha) = \alpha \varphi(D^{-1}\sigma(D))$ and
	$$
	\begin{aligned}
		\varphi^n \circ \sigma_{\operatorname{naive},\beta} \circ \varphi^{-n}(\beta) &= \varphi^n \circ \sigma_{\operatorname{naive},\beta} \circ \varphi^{-n}(\alpha D^{-1}) \\ &=  \alpha \varphi^{n+1}\left( D^{-1} \sigma(D) \right) \sigma(D^{-1}) \\ &= \beta D \varphi^{n+1}\left( D^{-1} \sigma(D) \right) \sigma(D^{-1}).
	\end{aligned}
	$$
	It follows from Theorem~\ref{galois}that $\operatorname{lim}_{n\rightarrow \infty} \varphi^n(D^{-1}\sigma(D))$ converges to a matrix $D_\sigma$ with entries in $\mathbb{F}$ and $\sigma(\alpha) = \alpha D_\sigma$. Write $\overline{x}$ for the image of $x$ under the reduction map $\mathcal{O}_{C^\flat} \otimes_{\mathbb{F}_p} \mathbb{F} \rightarrow \overline{k} \otimes_{\mathbb{F}_p} \mathbb{F}$. Observe that, for $x \in C^\flat \otimes_{\mathbb{F}_p} \mathbb{F}$, convergence of $\varphi^n(x)$ is equivalent to asking that $x \in \mathcal{O}_{C^\flat} \otimes_{\mathbb{F}_p} \mathbb{F}$ and $\overline{x} \in \mathbb{F}$. Furthermore, $\operatorname{lim}_{n\rightarrow \infty} \varphi^n(x)$ equals the image of $\overline{x}$ under the multiplicative section of this reduction map. Therefore, $D^{-1}\sigma(D) \in \operatorname{Mat}(\mathcal{O}_{C^\flat} \otimes_{\mathbb{F}_p} \mathbb{F})$ and so
	$$
	D_\sigma \equiv \begin{pmatrix}
		\frac{\sigma(d_1)}{d_1} & \sigma(d)\frac{\sigma(d_1)}{d_1}  - d \frac{\sigma(d_2)}{d_2}\\ 0 & \frac{\sigma(d_2)}{d_2}
	\end{pmatrix}	 \text{ modulo }\mathfrak{m}_{C^\flat} \otimes_{\mathbb{F}_p} \mathbb{F}.
	$$ 
	The next step is to give a simpler formula for $D_\sigma$ when $\sigma \in G_L$. In fact, in this case we claim that
	\begin{equation}\label{dsigma}
		D_\sigma  \equiv \begin{pmatrix}
			1 & \sigma(d) - d  \\ 0 & 1
		\end{pmatrix}  \text{ modulo } \mathfrak{m}_{C^\flat} \otimes_{\mathbb{F}_p} \overline{\mathbb{F}}_p.
	\end{equation}
	To show this it suffices to show that $\sigma(d_i)d_i^{-1} \equiv 1$ modulo $u^{ep/(p-1)} \mathcal{O}_{C^\flat}\otimes_{\mathbb{F}_p} \mathbb{F}$ for $\sigma  \in G_L$ and that $d \in u^{-(e+p-1)/(p-1)}\mathfrak{m}_{C^\flat} \otimes_{\mathbb{F}_p} \mathbb{F}$. Clearly the first claim implies the required congruences on the diagonal. When combined with the second claim it also implies that $\sigma(d)\frac{\sigma(d_1)}{d_1}  - d \frac{\sigma(d_2)}{d_2} \equiv \sigma(d) - d$ modulo $\mathfrak{m}_{C^\flat} \otimes_{\mathbb{F}_p} \mathbb{F}$, which establishes \eqref{dsigma}
	
	For the first claim note: if $\sigma \in G_L$, then $\frac{\sigma(d_i)}{d_i} -1$ is divisible by $\epsilon-1$, which we've already seen generates the ideal $u^{ep/(p-1)}\mathcal{O}_{C^\flat}$. For the second claim we show that $\overline{r}^\vee =(\mathfrak{M} \otimes_{k[[u]]} C^\flat)^{\varphi=1}$ is contained in $\mathfrak{M} \otimes_{k[[u]]} u^{-(e+p-1)/(p-1)}\mathfrak{m}_{C^\flat}$. This will imply that $d_1d \in u^{-(e+p-1)/(p-1)}\mathfrak{m}_{C^\flat} \otimes_{\mathbb{F}_p} \mathbb{F}$. Since $d_1$ is a unit multiple of $(v^{-\Omega_{\tau,t}})_\tau$ and $\Omega_{\tau,t} \geq 0$, it will follow that $d \in u^{-(e+p-1)/(p-1)}\mathfrak{m}_{C^\flat} \otimes_{\mathbb{F}_p} \mathbb{F} $ as desired. Take $x \in \overline{r}^\vee$. Then $x = (u^{n_\kappa})m$ for some $m \in\mathfrak{M} \otimes_{k[[u]]} \mathcal{O}_{C^\flat}$ and some $n_\kappa \in \mathbb{Q}$. Assume the $n_\kappa$ are maximal. Since $\varphi(x)=x$ it follows that $(u^{n_{\kappa}})m = (u^{pn_{\kappa\circ \varphi}})\varphi(m)$ and so $\varphi(m) = (u^{n_\kappa - pn_{\kappa \circ \varphi}})m$. Recall there exists $q_\tau \leq e+p-1$ with at least one inequality strict and $(u^{q_\kappa})\mathfrak{M}\subset \mathfrak{M}^\varphi$ (where $\mathfrak{M}^\varphi$ denotes the image of the linearised Frobenius). Therefore
	$$
	\varphi(m) = (u^{n_\kappa - pn_{\kappa \circ \varphi} - q_\kappa}) \varphi(m') = \varphi((u^{(n_\kappa - pn_{\kappa \circ \varphi} -q_\kappa)/p})m')
	$$
	for some $m' \in \mathfrak{M} \otimes_{k[[u]]} \mathcal{O}_{C^\flat}$. Injectivity of $\varphi$ ensures that $m \in (u^{(n_\kappa - pn_{\kappa \circ \varphi} -q_\kappa)/p}) \mathfrak{M} \otimes_{k[[u]]} \mathcal{O}_{C^\flat}$. Therefore,
	$$
	x \in (u^{n_\kappa + p^{-1}(n_\kappa - pn_{\kappa \circ \varphi} -q_\kappa)}) \mathfrak{M} \otimes_{k[[u]]} \mathcal{O}_{C^\flat}
	$$
	and so 
	$$
	n_\kappa - pn_{\kappa \circ \varphi} \leq q_\kappa,
	$$
	for all $\kappa$, since otherwise we contradict the maximality of the $n_\kappa$. Therefore,
	$$
	\begin{aligned}
		(1-p^{f})n_\kappa &= n_\kappa - pn_{\kappa \circ \varphi} + p(n_{\kappa \circ \varphi} - p n_{\kappa \circ \varphi^2}) + \ldots + p^{f-1}(n_{\kappa \circ \varphi^{f-1}} - p n_{\kappa}) \\
		&\leq q_\kappa + pq_{\kappa \circ \varphi} +\ldots + p^{f-1}q_{\kappa \circ \varphi^{f-1}} \\
		&< (e+p-1)\left( 1 + p+\ldots+ p^{f-1} \right) \\
		&= (e+p-1)\left( \frac{p^{f}-1}{p-1} \right).
	\end{aligned}
	$$
	We conclude that $n_\kappa > -(e+p-1)/(p-1)$ and so $x \in \mathfrak{M} \otimes_{k[[u]]} u^{-(e+p-1)/(p-1)}  \mathfrak{m}_{C^\flat}$ as desired. 
	
	Since $\sigma_{a,0} \in W^{\operatorname{exp}}(\overline{r}^{\operatorname{ss}})$ by Corollary~\ref{crimpliesexpinsemisimplecase}, we only have to show that the homomorphism $G_L \rightarrow \mathbb{F}$
	$$
	c: \sigma \mapsto \sigma(d) -d \text{ modulo } \mathfrak{m}_{C^\flat} \otimes_{\mathbb{F}_p} \mathbb{F}.
	$$
	is contained in $\Psi_{\sigma}(\chi_1,\chi_2)$ -- we know already that this homomorphism is obtained as the restriction of a cocycle in $G_K \rightarrow \mathbb{F}(\chi)$ and so $\operatorname{Gal}(L/K)$ acts by $\chi^{-1}$. By part (1) of Corollary~\ref{Abrcor}  we have
	$$
	c \in \Psi\left( (v^{\Omega_{\tau, t-s} + (p^f-1)(\delta_\tau- s_\tau)})_\tau l[[u]] \otimes_{\mathbb{F}_p} \mathbb{F}  \right).
	$$
	As in Corollary~\ref{crimpliesexpinsemisimplecase} take $J = \lbrace \tau \in \operatorname{Hom}_{\mathbb{F}_p}(k,\overline{\mathbb{F}}_p) \mid t_\tau < a_\tau+1 \rbrace$ and set $x_\tau = s_\tau -a_\tau-1$ for $\tau \in J$ and $x_\tau = s_\tau$ for $\tau \not\in J$. Then $x_\tau \in [0,e-1]$ and
	\begin{equation}\label{omegatauinsandt}
		\Omega_{\tau,t-s}  = \sum_{i=0}^{f-1} p^i\left( a_{\tau \circ \varphi^i} + e- 2s_{\tau \circ \varphi^i}\right) = \sum_{i=0}^{f-1}p^i (a_{\tau \circ \varphi^i}+1)(-1)^{\tau \circ \varphi^i \in J} + \sum_{i=0}^{f-1} p^i (e -1 -2 x_{\tau \circ \varphi^i}),	
	\end{equation}
	where $(-1)^{\tau \in J}$ equals $-1$ for $\tau \in J$ and $1$ otherwise. Notice this is exactly $\Omega_{\tau,\sigma,J,x}$ from Definition~\ref{subspacedefn}. Since $\delta_\tau -s_\tau = -x_\tau$, we conclude that $c \in \Psi_{\sigma,J,x} \subset H^1(G_L,\overline{\mathbb{F}}_p)$ for $\Psi_{\sigma,J,x}$ defined in Definition~\ref{subspacedefn}. We saw in the proof of Corollary~\ref{crimpliesexpinsemisimplecase} that
	$$
	\overline{r}^{\operatorname{ss}} \cong  \begin{pmatrix}
		\prod_{\tau \in J} \omega_{\tau}^{a_\tau + 1 + x_{\tau}} \prod_{\tau \not\in J} \omega_{\tau}^{x_{\tau}} & 0 \\ 
		0 & \prod_{\tau \not\in J} \omega_{\tau}^{a_\tau + e - x_{\tau}} \prod_{\tau \in J} \omega_{\tau}^{e-1 - x_{\tau}}
	\end{pmatrix}.
	$$ Therefore $c \in \Psi_{\sigma_{a,0}}(\chi_1,\chi_2)$ and so $\sigma_{a,0} \in W^{\operatorname{exp}}(\overline{r})$.
\end{proof}
\section{Proof of the main theorem}\label{finishproof}

We are now ready to put together the proof of Theorem~\ref{main}. As usual write $\overline{r} = \left( \begin{smallmatrix}
	\chi_1 & c \\ 0 & \chi_2
\end{smallmatrix}\right)$ and set $\chi = \chi_1/\chi_2$. First, the following lemma shows that it suffices to prove $\sigma_{a,0} \in W^{\operatorname{exp}}(\overline{r})$ if and only if $\sigma_{a,0} \in W^{\operatorname{cr}}(\overline{r})$ for any $\overline{r}$.
\begin{lem}\label{lem-twist}
	For $* \in \lbrace \operatorname{cr},\operatorname{exp} \rbrace$, we have $\sigma_{a-b,0} \in W^{*}(\overline{r})$ if and only if $\sigma_{a,b} \in W^*(\overline{r} \otimes \prod_\tau \omega_\tau^{ b\tau})$.
\end{lem}
\begin{proof}
	For $* =\operatorname{exp}$, this follows immediately from the definitions. For $* = \operatorname{cr}$, choose a labelling $\tau_0,\ldots,\tau_{e-1}$ of $\lbrace \kappa \in \operatorname{Hom}_{\mathbb{Q}_p}(K,\overline{\mathbb{Q}}_p) \mid \kappa|_k = \tau \rbrace$ as in Definition~\ref{exp}. It is well known that there exists a crystalline character $\widetilde{\chi}:G_K \rightarrow \overline{\mathbb{Z}}_p^\times$ whose reduction modulo $\mathfrak{m}_{\overline{\mathbb{Z}}_p}$ is $\prod_\tau \omega_\tau^{b_\tau}$ and with 
	$$
	\operatorname{HT}_{\tau_i}(\widetilde{\chi}) = \begin{cases}
		\lbrace b_\tau \rbrace & \text{if $i =0$;} \\
		\lbrace 0 \rbrace & \text{if $i = 1,\ldots,e-1$.}
	\end{cases}
	$$
	Thus, if $r$ is a crystalline lift witnessing $\sigma_{a-b,0} \in W^{\operatorname{cr}}(\overline{r})$, then $r \otimes \widetilde{\chi}$ is a crystalline lift witnessing $\sigma_{a,b} \in W^*(\overline{r} \otimes \prod_\tau \omega_\tau^{ b\tau})$.
\end{proof}

Next we consider the degenerate situations from Lemma~\ref{lem-degen} and Corollary~\ref{cor-degen}. 

\begin{lem}
	If $\overline{r}= \psi \otimes \left( \begin{smallmatrix}
		\chi_{\operatorname{cyc}} & c \\0 & 1
	\end{smallmatrix} \right)$ for an unramified character $\psi$ and $\sigma = \sigma_{a,0}$ with $a_\tau =p-1$ for every $\tau$, then $\sigma \in W^{\operatorname{exp}}(\overline{r})$ if and only if $\sigma \in W^{\operatorname{cr}}(\overline{r})$.
\end{lem}
\begin{proof}
	Corollary~\ref{cor-degen} showed that $\sigma \in W^{\operatorname{exp}}(\overline{r})$ so we need to show $\sigma \in W^{\operatorname{cr}}(\overline{r})$ by producing a crystalline lift of $\overline{r}$ of Hodge type $\sigma_{a,0}$. Below we sketch the well known construction of such a lift following  \cite[9.4]{gls14} (which treats the unramified case) and \cite[5.2.9]{gls12} (which treats the totally ramified case). 
	
	Choose an indexing $\tau_0,\ldots,\tau_{e-1}$ of those embeddings $\kappa:K \rightarrow \overline{\mathbb{Q}}_p$ with $\kappa|_k = \tau$. Since $\chi_{\operatorname{cyc}}|_{I_K} = \prod_{\tau} \omega_{\tau}^{e+p-1}$, there exists a crystalline character $\widetilde{\chi}:G_K \rightarrow \overline{\mathbb{Z}}_p^\times$ lifting $\psi \chi_{\operatorname{cyc}}$ with 
	$$
	\operatorname{HT}_{\tau_i}(\widetilde{\chi}) = \begin{cases}
		\lbrace p \rbrace & \text{if $i =0$;} \\
		\lbrace 1 \rbrace & \text{if $i = 1,\ldots,e-1$.}
	\end{cases}
	$$
	For any unramified character $\widetilde{\psi}$ with $\widetilde{\psi} \equiv 1$ modulo $\mathfrak{m}_{\overline{\mathbb{Z}}_p}$, consider the Block--Kato subgroup 
	$$
	H^1_f(G_K,\overline{\mathbb{Z}}_p(\widetilde{\psi}\widetilde{\chi} )) \subset H^1(G_K,\overline{\mathbb{Z}}_p(\widetilde{\psi} \widetilde{\chi}))
	$$
	classifying crystalline extensions of $1$ by $\widetilde{\psi} \widetilde{\chi}$.  Any such extension has Hodge type $\sigma_{a,0}$, so we will be done if we can show that any class in $H^1(G_K, \overline{\mathbb{F}}_p(\psi \chi_{\operatorname{cyc}}))$ is contained in the image of the reduction map
	$$
	H^1_f(G_K,\overline{\mathbb{Z}}_p(\widetilde{\psi} \widetilde{\chi})) \rightarrow H^1(G_K,\overline{\mathbb{F}}_p(\psi \chi_{\operatorname{cyc}}) )
	$$
	for at least one $\widetilde{\psi}$. In fact, since every Hodge--Tate weight of $\widetilde{\chi}$ is $\geq 0$, one has $H^1_f(G_K,\overline{\mathbb{Z}}_p(\widetilde{\psi}\widetilde{\chi} )) = H^1(G_K,\overline{\mathbb{Z}}_p(\widetilde{\psi}\widetilde{\chi}))$. Therefore, this can be checked using standard techniques from Galois cohomology.
\end{proof}

The previous lemma allows us to assume that if $\sigma =\sigma_{a,0}$ with $a_\tau =p-1$ for every $\tau$, then $\overline{r} \neq \psi \otimes \left( \begin{smallmatrix}
	\chi_{\operatorname{cyc}} & c \\0 & 1
\end{smallmatrix} \right)$ for some unramified character. Therefore, Corollary~\ref{ASfinal} applies and we are left proving that $\sigma_{a,0} \in W^{\operatorname{exp}}(\overline{r})$ implies $\sigma_{a,0} \in W^{\operatorname{cr}}(\overline{r})$. To do this we again have to exhibit a crystalline lift $r$ of $\overline{r}$ of Hodge type $\sigma_{a,0}$ and we again produce this $r$ as an extension of two carefully chosen crystalline characters. Since $\sigma_{a,0} \in W^{\operatorname{exp}}(\overline{r})$, there is a maximal pair $(J_{\operatorname{max}},x_{\operatorname{max}})$ as in Proposition~\ref{proposition-maximal} so that
$$
\overline{r} = \begin{pmatrix} 
	\psi_1 \prod_{\tau \in J_{\operatorname{max}}} \omega_{\tau}^{a_\tau + 1 + x_{\operatorname{max},\tau}} \prod_{\tau \not\in J_{\operatorname{max}}} \omega_{\tau}^{ x_{\operatorname{max},\tau}} & c \\ 
	0 & \psi_2 \prod_{\tau \not\in J_{\operatorname{max}}} \omega_{\tau}^{a_\tau + e - x_{\operatorname{max},\tau}} \prod_{\tau \in J_{\operatorname{max}}} \omega_{\tau}^{ e-1 - x_{\operatorname{max},\tau}}
\end{pmatrix}
$$
and $c \in \Psi_{\sigma}(\chi_1,\chi_2)^{\operatorname{Gal}(L/K) = \chi^{-1}}$. To produce the crystalline lift of $\overline{r}$ choose an indexing $\tau_0,\ldots,\tau_{e-1}$ of those embeddings $\kappa:K\rightarrow \overline{\mathbb{Q}}_p$ with $\kappa|_{k} = \tau$. We consider crystalline extensions
$$
r = \begin{pmatrix}
	\widetilde{\chi}_1 & C \\ 0 & \widetilde{\chi}_2
\end{pmatrix}
$$
for crystalline characters  $\widetilde{\chi}_1$ and $\widetilde{\chi}_2$ with $\tau_0$-Hodge--Tate weights
$$
(\operatorname{HT}_{\tau_0}(\widetilde{\chi}_1), \operatorname{HT}_{\tau_0}(\widetilde{\chi}_2)) = \begin{cases}
	(a_\tau+1,0) & \text{ if $\tau \in J_{\operatorname{max}}$;} \\
	(0,a_\tau+1) & \text{ if $\tau \not\in J_{\operatorname{max}}$.}
\end{cases}
$$
For the other embeddings we require $(\operatorname{HT}_{\tau_j}(\widetilde{\chi}_1), \operatorname{HT}_{\tau_j}(\widetilde{\chi}_2))$ equal $(1,0)$ for $j =1,\ldots,x_{\operatorname{max},\tau}$ and equal $(0,1)$ for $j = x_{\operatorname{max},\tau}+1,\ldots,e-1$. Then we have $\widetilde{\chi}_1|_{I_K} \cong \prod_{\tau \in J_{\operatorname{max}}} \omega_{\tau}^{a_\tau + 1 + x_{\operatorname{max},\tau}} \prod_{\tau \not\in J_{\operatorname{max}}} \omega_{\tau}^{x_{\operatorname{max,\tau}}} \bmod{\mathfrak{m}_{\overline{\mathbb{Z}}_p}}$. Thus, replacing $\widetilde{\chi}_1$ by an unramified twist we can further assume $\widetilde{\chi}_1$ lifts $\chi_1$. Similarly, we can assume $\widetilde{\chi}_2$ lifts $\chi_2$.

Any such extension $r$ has Hodge type $\sigma_{a,0}$ and the cocycles $C$ defining such an extension are described by the Bloch--Kato subspace $H^1_f(G_K,\overline{\mathbb{Z}}_p(\widetilde{\chi}_1 \widetilde{\chi}_2^{-1})) \subset H^1(G_K,\overline{\mathbb{Z}}_p(\widetilde{\chi}_1 \widetilde{\chi}_2^{-1}))$. Let $Q'$ denote the image of this Bloch--Kato subspace under $H^1(G_K,\overline{\mathbb{Z}}_p(\widetilde{\chi}_1 \widetilde{\chi}_2^{-1})) \rightarrow H^1(G_K, \overline{\mathbb{F}}_p( \chi))$. We claim 
$$
\operatorname{dim}_{\overline{\mathbb{F}}_p} Q' = \nu' + \sum_\tau \begin{cases}
	x_{\operatorname{max},\tau} +1 & \text{if $\tau \in J_{\operatorname{max}}$} \\
	x_{\operatorname{max},\tau}  & \text{if $\tau \not\in J_{\operatorname{max}}$}
\end{cases} = \nu' + \operatorname{Card}(J_{\operatorname{max}}) +  \sum_\tau x_{\operatorname{max},\tau},
$$
where $\nu' =0$ unless $\chi =1$ in which case $\nu' =1$. To see this note this dimension is the sum of the dimension of the $p$-torsion in $H^1(G_K,\overline{\mathbb{Z}}_p(\widetilde{\chi}_1 \widetilde{\chi}_2^{-1}))$ (which is $\nu'$) and the $\overline{\mathbb{Q}}_p$-dimension of $H^1(G_K,\overline{\mathbb{Q}}_p(\widetilde{\chi}_1 \widetilde{\chi}_2^{-1}))_f$. It follows from \cite[1.24]{Nek93} that this latter $\overline{\mathbb{Q}}_p$-dimension is precisely the number $\kappa: K \rightarrow \overline{\mathbb{Q}}_p$ for which the $\kappa$-Hodge--Tate weight of $\widetilde{\chi}_1$ is greater than the $\kappa$-Hodge--Tate weight of $\widetilde{\chi}_2$. Examining the Hodge--Tate weights of $\widetilde{\chi}_1$ and $\widetilde{\chi}_2$, we see this number is precisely the sum in the second part of the claimed formula.

Write $Q$ for the image of $Q'$ under the injection $H^1(G_K,\overline{\mathbb{F}}_p(\chi)) \hookrightarrow H^1(G_L,\overline{\mathbb{F}}_p)$. It follows from Corollary~\ref{ASfinal} that any element of $Q$ is contained in $\Psi_{\sigma}(\chi_1,\chi_2)^{\operatorname{Gal}(L/K) = \chi^{-1}}$. Corollary~\ref{cor-maxpairs} implies that the dimension of $Q$ is at least the dimension of $\Psi_{\sigma}(\chi_1,\chi_2)^{\operatorname{Gal}(L/K) = \chi^{-1}}$ since $\nu = 1$ from Corollary~\ref{cor-maxpairs} implies the $\nu'$ defined above equals $1$. Therefore, $Q = \Psi_{\sigma}(\chi_1,\chi_2)^{\operatorname{Gal}(L/K) = \chi^{-1}}$ and we can choose $C$ so that $\overline{r} = r \otimes_{\overline{\mathbb{Z}}_p} \overline{\mathbb{F}}_p$ as desired. This finishes the proof of Theorem~\ref{main}.

It follows from these results that Corollary~\ref{cor-maxpairs} can be improved as follows.

\begin{cor}\label{dimPsiequals}
 	If $(J_{\operatorname{max}},x_{\operatorname{max}})$ is the maximal pair from part (2) of Proposition~\ref{proposition-maximal} and $\chi = \chi_1/\chi_2$, then
	$$
	\Psi_{\sigma}(\chi_1,\chi_2) = \Psi_{\sigma,J_{\operatorname{max}},x_{\operatorname{max}}}
	$$
	and
	$$
	\begin{aligned}
		\operatorname{dim}_{\overline{\mathbb{F}}_p} \Psi_{\sigma}(\chi_1,\chi_2)^{\operatorname{Gal}(L/K) = \chi^{-1}} = \nu' + \operatorname{Card}(J_{\operatorname{max}})  + \sum_{\tau}x_{\operatorname{max},\tau}, 
	\end{aligned}
	$$
	where $\nu' =0$ unless $\chi =1$ in which case $\nu'=1$.
\end{cor}

\section{Explicit comparison with Demb\`{e}l\`{e}--Diamond--Roberts}\label{sec-DDR}

When $K/\mathbb{Q}_p$ is unramified \cite{ddr16} define an alternative explicit set of weights $W^{\operatorname{DDR}}(\overline{r})$ using local class field theory. The remainder of this paper will be devoted to proving the following theorem.

\begin{prop}\label{DDR=us} Suppose $p>2$ and $K/\QQp$ is unramified. Then $W^\mathrm{exp}(\ovr r) = W^\mathrm{DDR}(\ovr r)$.
\end{prop}

As mentioned in the introduction, this follows from the results above and those of \cite{cegm17}, since both sets have the same description in terms of crystalline lifts. In the spirit of this paper, we will instead give a direct proof of the equality using a reciprocity law of Br\"uckner--Shaferevich--Vostokov (see \cite[Thm.~4]{vos79}) without reference to any $p$-adic Hodge theory. As a consequence we get an alternative proof of the conjecture of \cite{ddr16} when $p>2$.

We begin by recalling the description of $W^{\operatorname{DDR}}(\overline{r})$. For this we can suppose $\ovr r\colon G_K \to \GL_2(\FFpb)$ is reducible (when $\ovr r$ is irreducible $W^\mathrm{exp}(\ovr r) = W^\mathrm{DDR}(\ovr r)$ is essentially true by definition). As before, we write
\[
\ovr r \sim \begin{pmatrix} \chi_1 & c \\ 0 & \chi_2 \end{pmatrix}
\]
for characters $\chi_1,\chi_2\colon G_K\to \GL_2(\FFpb)$. Set $\chi:=\chi_1\chi_2^{-1}$ and write
\[
\chi = \psi\prod_{\tau \in \Hom_{\FFp}(k, \FFpb)} \omega_\tau^{a_\tau},
\]
where $\psi$ is an unramified character and $a_\tau \in [1,p]$ with $a_\tau<p$ for at least one $\tau$. Recall that this uniquely determines the $a_\tau$. For a fixed $\tau\colon k \hookrightarrow \FFpb,$ we also let $\lambda_{\tau, \psi}$ denote a basis of the one-dimensional $\FFpb$-vector space $(l\otimes_{k,\tau} \FFpb)^{\Gal(L/K) = \psi}.$ As before, let $\pi$ denote a uniformiser of $K$ and let $\pi^{1/(p^f-1)}$ denote a $(p^f-1)$-th root of $\pi$ in a fixed algebraic closure.

\begin{con}
	Write $\varpi:=\pi^{1/(p^f-1)}$ and consider the homomorphism
	\begin{align*}
		\varepsilon_{\varpi^r}\colon l \otimes_{\FFp} \FFpb &\to \O_L^\times \otimes_{\ZZp} \FFpb; \\
		a \otimes b &\mapsto E^\mathrm{AH}([a]\varpi^r) \otimes b.
	\end{align*}
	with $E^{\operatorname{AH}}$ as defined in Section~\ref{AH}. In \cite{ddr16} an explicit basis of $H^1(G_K, \FFpb(\chi))$ is defined as follows. For each $\tau\colon k \hookrightarrow \FFpb$, we will define an embedding $\tau'$ and an integer $n_\tau'.$ Recall the definition of $\Omega_{\tau,a}$ from (\ref{omegadef}). If $a_{\tau \circ \varphi} \neq p$, then 
	\begin{itemize}
		\item $\tau' := \tau \circ \varphi$ and $n_\tau ' = \Omega_{\tau \circ \varphi, a}.$
	\end{itemize} 
	However, if $a_{\tau \circ \varphi} = p$, then let $j$ equal the smallest integer $>1$ with $a_{\tau \circ \varphi^j} \neq p-1$ and set
	\begin{itemize}
		\item $\tau':= \tau \circ \varphi^j$ and $n_\tau ' = \Omega_{\tau \circ \varphi^j, a} - (p^f - 1).$
	\end{itemize}
	Then we define
	\[
	u_\tau:=\varepsilon_{\varpi^{n_\tau '}}(\lambda_{\tau', \psi}) \in \O_L^\times \otimes_{\ZZ} \FFpb,
	\]
	for all $\tau \in \Hom_{\FFp}(k, \FFpb).$ If $\chi = 1$, we additionally define $u_\mathrm{triv} := \varpi \otimes 1 \in \O_L^\times \otimes_{\ZZ} \FFpb$. If $\chi$ is cyclotomic, we additionally define $u_\mathrm{cyc}:= \varepsilon_{\varpi^{p(p^f-1)/(p-1)}}(b \otimes 1),$ where $b\in l$ is any element with $\mathrm{Tr}_{l/\FFp}(b) \neq 0.$
\end{con}

\begin{lem}\label{basis}
	The elements $\{u_\tau \mid \tau \in \Hom_{\FFp}(k, \FFpb)\}$, together with $u_\mathrm{triv}$ if $\chi$ is trivial and $u_\mathrm{cyc}$ if $\chi$ is cyclotomic, forms a basis of the $\FFpb$-vector space
	\[
	U_\chi:=\left(L^\times \otimes \FFpb\right)^{\Gal(L/K) = \chi}.
	\]
\end{lem}
\begin{proof}
See \cite[Theorem 5.1]{ddr16}.
\end{proof}

The isomorphism $G_L^{\operatorname{ab}} \cong \widehat{L}^\times$ of local class field theory induces an identification
$$
H^1(G_L,\overline{\mathbb{F}}_p) = \operatorname{Hom}_{\overline{\mathbb{F}}_p}(L^\times \otimes_{\mathbb{Z}} \overline{\mathbb{F}}_p,\overline{\mathbb{F}}_p)
$$
under which $H^1(G_K,\overline{\mathbb{F}}_p(\chi)) = H^1(G_L,\overline{\mathbb{F}}_p)^{\operatorname{Gal}(L/K) = \chi^{-1}}$ identifies with the $\overline{\mathbb{F}}_p$-linear dual of $U_\chi$. Thus, we can define a subspace of $H^1(G_K,\overline{\mathbb{F}}_p(\chi))$ in terms of the vanishing of cocycles on certain elements of the basis above as follows.

\begin{defn}
	Fix a Serre weight $\sigma = \sigma_{a,b}$ and write $r_\tau := a_\tau-b_\tau+1$. Assume $\sigma  \in W^{\operatorname{exp}}(\overline{r}^{\operatorname{ss}})$. It follows that there exists a $J \subset \operatorname{Hom}_{\mathbb{F}_p}(k,\overline{\mathbb{F}}_p)$ with
	\begin{equation}\label{unramandsemisimple}
	\chi_1|_{I_K} = \prod_{\tau \in J} \omega_\tau^{a_\tau + 1} \prod_{\tau \not\in J} \omega_\tau ^{b_\tau}, \qquad \chi_2|_{I_K} = \prod_{\tau \not\in J} \omega_{\tau}^{a_\tau + 1} \prod_{\tau \in J} \omega_{\tau}^{b_\tau}
\end{equation} 
	and let $(J_{\operatorname{max}},x_{\operatorname{max}})$ be the maximal subset in the sense of Proposition~\ref{proposition-maximal}. (Recall we are assuming $e=1$ and we must have $x_\tau =0$ for all $\tau$ and similarly for $x_{\operatorname{max}}$.) In \cite[\S7.1]{ddr16} a cardinality-preserving shift function $\mu\colon \wp\left(\Hom_{\FFp}(k, \FFpb)\right) \to \wp\left(\Hom_{\FFp}(k, \FFpb)\right)$ is defined on subsets of $\operatorname{Hom}_{\mathbb{F}_p}(k,\overline{\mathbb{F}}_p)$. Then we define
	$$
	L_\sigma^\mathrm{DDR}(\chi_1,\chi_2)\subseteq H^1(G_K, \FFpb(\chi))
	$$
	to be the subspace consisting of those cocycles $f \in H^1(G_K, \FFpb(\chi))$ with
	\begin{itemize}
		\item $f(u_\tau) =0$ for all $\tau \not\in \mu(J_{\operatorname{max}})$.
		\item $f(u_{\operatorname{cyc}})= 0$ if $\chi = \chi_{\operatorname{cyc}}$ except when, additionally, $J_{\operatorname{max}} = \operatorname{Hom}_{\mathbb{F}_p}(k,\overline{\mathbb{F}}_p)$ and $r_\tau = p$ for all $\tau$, in which case we have no requirement at $u_{\operatorname{cyc}}$.
	\end{itemize}
	\begin{rem}
	    The shift function $\mu$ has a rather involved construction which, for the time being, we will not need. The only point where the actual definition of $\mu$ is used in the assertions preceding Remark~\ref{rem-wheremuisused}.
	\end{rem}
	In other words, if $c_\tau \in H^1(G_L,\overline{\mathbb{F}}_p)$ denotes the $\FFpb$-dual of $u_\tau$, and similarly for $c_{\operatorname{triv}}$ and $c_{\operatorname{cyc}}$, then $L_\sigma^{\operatorname{DDR}}(\chi_1,\chi_2)$ is the span of the $c_\tau$ for $\tau \in \mu(J_{\operatorname{max}})$, together with $c_{\operatorname{triv}}$ if $\chi =1$ and $c_{\operatorname{cyc}}$ if $\chi = \chi_{\operatorname{cyc}}$, $J_{\operatorname{max}} = \operatorname{Hom}_{\mathbb{F}_p}(k,\overline{\mathbb{F}}_p)$ and $r_\tau = p$ for all $\tau$.
\end{defn}

\begin{defn}\label{defn:W-DDR}
	For $\overline{r} \sim \left( \begin{smallmatrix} \chi_1 & c \\ 0 & \chi_2\end{smallmatrix}\right)$ as above, the set $W^{\operatorname{DDR}}(\overline{r})$ is defined as follows. We have $\sigma = \sigma_{a,b} \in W^{\operatorname{DDR}}(\overline{r})$ if and only if
	\begin{enumerate}
		\item $\sigma \in W^{\operatorname{exp}}(\overline{r}^{\operatorname{ss}})$ and
		\item $[c] \in L_\sigma^\mathrm{DDR}(\chi_1,\chi_2)$.
	\end{enumerate}
\end{defn}

\section{Vostokov's formula}\label{subsec:vost-formula}

In this section we recall the explicit reciprocity laws described in \cite{vos79} which requires $p> 2$. As in the previous section, local class field theory allows us to identify 
$$
H^1(G_L,\overline{\mathbb{F}}_p) = \operatorname{Hom}_{\overline{\mathbb{F}}_p}(L^\times\otimes_{\mathbb{Z}} \overline{\mathbb{F}}_p,\overline{\mathbb{F}}_p).
$$
Fix a primitive $p$-th root of unity $\epsilon_1 \in L$. Then, for $\alpha,\beta \in L^\times \otimes_{\mathbb{Z}} \overline{\mathbb{F}}_p^\times$, we write
$$
c(\alpha,\beta) := f_\beta(\alpha),
$$
where $f_\beta: G_L \rightarrow \overline{\mathbb{F}}_p$ denotes the image of $\beta$ under the Kummer map $L^\times \otimes_{\mathbb{Z}} \overline{\mathbb{F}}_p = H^1(G_L,\mu_p(L) \otimes_{\mathbb{F}_p} \overline{\mathbb{F}}_p) = H^1(G_L,\overline{\mathbb{F}}_p)$. Here the second equality comes from the identification $\mu_p(L) = \mathbb{F}_p$ induced by $\epsilon_1$, so that concretely $c(\alpha,\beta)$ is defined so that
$$
\sigma_\alpha(\beta^{1/p}) \beta^{-1/p} = \epsilon_1^{c(\alpha,\beta)}
$$
for any $p$-th root $\beta^{1/p}$ of $\beta$ and $\sigma_\alpha \in G_L$ any element mapped onto $\alpha$ by $G_L \rightarrow G_L^{\operatorname{ab}} \rightarrow \widehat{L}^\times$.
\begin{thm}\label{thm-vos}
	Let $L^{\operatorname{AH}}$ denote the inverse of $E^{\operatorname{AH}}: vW(l)[[v]] \xrightarrow{\sim} 1 + vW(l)[[v]]$ and let $z(v) \in W(l)[[v]]$ be such that $z(\varpi) = \epsilon_1$. For $A,B \in W(l)((v))^\times$, write
	$$
	A = v^a [\theta] \epsilon , \qquad B = v^b [\theta'] \eta
	$$
	with $\epsilon,\eta \in 1+vW(l)[[v]]$ and $\theta,\theta' \in l^\times$. Set
	\begin{equation}\label{eqn:vost-formula} 
		\gamma = \mathrm{res}_v\left(\left(L^\mathrm{AH}(\varepsilon(v))\frac{d L^{\mathrm{AH}}(\eta(v))}{dv} - L^{\mathrm{AH}}(\varepsilon(v))d_{\mathrm{log}}(B(v)) + L^{\mathrm{AH}}(\eta(v)) d_{\mathrm{log}}(A(v))\right)\left(\frac{1}{z(v)^p-1}\right)\right),
	\end{equation}
	where $\operatorname{res}_v(x)$ denotes the coefficient of $v^{-1}$ in $x \in W(l)((v))$ and $d_{\operatorname{log}}$ denotes the logarithmic derivative $\frac{1}{x(v)}\frac{d}{dv}x(v)$. Then
	$$
	c(A(\varpi),B(\varpi)) = \operatorname{Tr}_{W(l)/\mathbb{Z}_p}(\gamma) \text{ modulo }p.
	$$
\end{thm}
\begin{proof}
	This is \cite[Theorem 4.]{vos79}.
\end{proof}

The following form of this result is what we will use in our proofs.

\begin{cor}\label{cor-vos}
Also write $E^{\operatorname{AH}}$ for the base-change of the isomorphism $vW(l)[[v]] \xrightarrow{\sim} 1 + vW(l)[[v]]$ along $\otimes_{\mathbb{Z}_p} \overline{\mathbb{F}}_p$. Then, for $x,y(z(v)^p-1) \in vl[[v]] \otimes_{\mathbb{F}_p} \overline{\mathbb{F}}_p$, we have
$$
c(E^{\operatorname{AH}}(x)|_{v = \varpi}, E^{\operatorname{AH}}(y(z(v)^p-1))|_{v=\varpi}) = \operatorname{Tr}_{l \otimes_{\FFp} \FFpb/ \FFpb}( \overline{\gamma}_0), 
$$
where $\overline{\gamma}_0$ denotes the constant term\footnote{By constant term we mean the image of the Laurent series under the $\FFpb$-linear extension of the $\FFp$-linear map $l((v))\to l$ sending the series onto its constant term. After identifying $l((v))\otimes_{\FFp}\FFpb = \prod_\tau \FFpb((v))$, this becomes the constant term in each coordinate.} of
$$
\sum_{m \geq 0} \left( y\varphi^m \left( v\frac{d}{dv}(x) \right)\right) - \sum_{m \geq 1}  \left(\frac{1}{z(v)^p-1}\right)\left(  x \varphi^m \left( v\frac{d}{dv}(y(z(v)^p-1)) \right)\right). 
$$
\end{cor}
\begin{proof}
First suppose $x,y \in vl[[v]]$. Then Theorem~\ref{thm-vos} implies the assertion with $\overline{\gamma}_0$ replaced by
$$
\operatorname{res}_v\left( \left( \frac{1}{z(v)^p-1}\right)\left(x \frac{d}{dv}((z(v)^p-1)y) - x d_{\operatorname{log}}\left(E^{\operatorname{AH}}(y(z(v)^p-1))\right) + (z(v)^p-1)y d_{\operatorname{log}}\left(E^{\operatorname{AH}}(x)\right) \right) \right).
$$ 
We compute
$$
\begin{aligned}
\operatorname{d}_{\operatorname{log}}\left(E^{\operatorname{AH}}(x)\right) &= d_{\operatorname{log}} \operatorname{exp}\left( \sum_{m\geq 0} \frac{\varphi^m}{p^m}(x)\right) = \frac{d}{dv}\left( \sum_{m\geq 0} \frac{\varphi^m}{p^m}(x)\right) =  v^{-1}\sum_{m\geq 0} \varphi^m \left( v\frac{d}{dv}(x)\right),
\end{aligned}
$$
where the last equality follows from the identity $v\frac{d}{dv} \circ \varphi^m = p^m \varphi^m \circ v\frac{d}{dv}$. Using this we can rewrite the above residue as the residue of 
$$
\begin{aligned}
\left(\frac{1}{z(v)^p-1}\right) \left( x \frac{d}{dv}(y(z(v)^p-1)) -x \left( v^{-1}\sum_{m \geq 0} \varphi^m \left( v\frac{d}{dv}(y(z(v)^p-1))\right)\right) \right) + y \left( v^{-1}\sum_{m \geq 0} \varphi^m \left( v\frac{d}{dv}(x) \right)  \right) \\
 = \left(\frac{-x}{z(v)^p-1}\right) \left( v^{-1} \sum_{m\geq 1}  \varphi^m \left( v \frac{d}{dv}(y(z(v)^p-1))\right) \right) + v^{-1}\sum_{m \geq 0} \left( y \varphi^m \left( v \frac{d}{dv}(x)\right) \right).
\end{aligned}
$$
This is precisely the constant term of $\overline{\gamma}_0$. It follows that the identity holds for general $x,y \in vl[[v]] \otimes_{\FFp} \overline{\mathbb{F}}_p$ by linearity.
\end{proof}

\section{Proof of Proposition~\ref{DDR=us}}\label{sec:explicit-comp}

Suppose $\overline{r} = \left( \begin{smallmatrix} \chi_1 & c \\ 0 & \chi_2 \end{smallmatrix} \right)$ and write $\chi = \chi_1/\chi_2$. To prove $W^{\operatorname{exp}}(\overline{r})= W^{\operatorname{DDR}}(\overline{r})$, it suffices to show that
$$
L^{\operatorname{DDR}}_\sigma(\chi_1,\chi_2) = \Psi_{\sigma}(\chi_1,\chi_2)^{\operatorname{Gal}(L/K) = \chi^{-1}}
$$
for each Serre weight $\sigma$ with $\sigma \in W^{\operatorname{exp}}(\overline{r}^{\operatorname{ss}})$. Since the analogue of Lemma~\ref{lem-twist} also holds with $* =\operatorname{DDR}$, we can assume $\sigma =\sigma_{a,0}$. Then $\sigma  \in W^{\operatorname{exp}}(\overline{r}^{\operatorname{ss}})$ implies the existence of $J \subset \operatorname{Hom}_{\mathbb{F}_p}(k,\overline{\mathbb{F}}_p)$ so that \eqref{unramandsemisimple} holds. Write $(J_{\operatorname{max}}, x_{\operatorname{max}})$ for the maximal such subset from Proposition~\ref{proposition-maximal}. Since $e=1$, we have that $x_{\operatorname{max}, \tau} = 0$ for all $\tau$. Let $\psi$ be the unramified character so that $\chi^{-1} = \psi \omega_{J_{\operatorname{max}},x_{\operatorname{max}}}$; recall $\omega_{J_{\operatorname{max}},x_{\operatorname{max}}}$ is the character from Remark~\ref{omegawhy}.

\subsubsection*{The degenerate case}
First, we treat the case where $J_{\operatorname{max}} = \operatorname{Hom}_{\mathbb{F}_p}(k,\overline{\mathbb{F}}_p)$ and $a_\tau = p-1$ for each $\tau$. Then Lemma~\ref{lem-degen} implies $\Psi_{\sigma}(\chi_1,\chi_2)^{\operatorname{Gal}(L/K) = \chi^{-1}} = H^1(G_K,\overline{\mathbb{F}}_p(\chi))$. Note that $\mu$ is cardinality preserving, so that $\mu(J_{\operatorname{max}}) = \operatorname{Hom}_{\mathbb{F}_p}(k,\overline{\mathbb{F}}_p)$ in this case. Then Lemma~\ref{basis} implies $L^{\operatorname{DDR}}_{\sigma}(\chi_1,\chi_2) = H^1(G_K,\overline{\mathbb{F}}_p(\chi))$, so we are done in this case.

\subsubsection*{The non-degenerate cases}
For the rest of the proof we can assume that either $J_{\operatorname{max}} \neq \operatorname{Hom}_{\mathbb{F}_p}(k,\overline{\mathbb{F}}_p)$ or $a_\tau < p-1$ for some $\tau$. To simplify notations we set
$$
\Omega_\tau = \Omega_{\tau,\sigma,J_{\operatorname{max}},x_{\operatorname{max}}}
$$
and note that, due to the assumptions on $J_{\operatorname{max}}$ and $a_\tau$, we have $\Omega_\tau > -p(p^f-1)/(p-1)$.

Having excluded the degenerate case above, we note that it follows from the definitions that 
$$
\dim_{\FFpb}\left(L_\sigma^\mathrm{DDR}(\chi_1,\chi_2)\right) = \nu' + \mathrm{Card}(J_\mathrm{max}),
$$
where $\nu'=1$ if $\chi = 1$ and 0 otherwise. Corollary~\ref{dimPsiequals} therefore shows that $ \Psi_{\sigma}(\chi_1,\chi_2)^{\operatorname{Gal}(L/K) = \chi^{-1}}
$
and $L_\sigma^{\operatorname{DDR}}(\chi_1,\chi_2)$
have the same $\overline{\mathbb{F}}_p$-dimension. Therefore, it suffices to show that 
$$
\Psi_{\sigma}(\chi_1,\chi_2)^{\operatorname{Gal}(L/K) = \chi^{-1}} \subseteq L_\sigma^{\operatorname{DDR}}(\chi_1,\chi_2).
$$
For $\alpha \in L^\times$, let $\sigma_\alpha \in G_L$ be an element mapped onto $\alpha$ by $G_L \rightarrow G_L^{\operatorname{ab}} \rightarrow \widehat{L}^\times$. Then, since $\Psi_{\sigma}(\chi_1,\chi_2) = \Psi_{\operatorname{J}_{\operatorname{max},x_{\operatorname{max}}}}$, the value of an element in $\Psi_{\sigma}(\chi_1,\chi_2)$ at $\sigma_\alpha$ is computed by 
$$
c\left(\alpha, E^{\operatorname{AH}}((v^{\Omega_{\tau}})_\tau (z(v)^p-1)x)|_{v=\varpi}\right) 
$$
for some $x \in l[[u]] \otimes_{\mathbb{F}_p} \overline{\mathbb{F}}_p$. The definition of $L_\sigma^{\operatorname{DDR}}(\chi_1,\chi_2)$ therefore implies that the desired inclusion will follow from
\begin{enumerate}
	\item[(V1)] $c\left(E^{\operatorname{AH}}(v^{n_{\kappa'}} \lambda_{\kappa',\psi})|_{v = \varpi}, E^{\operatorname{AH}}((v^{\Omega_{\tau}})_\tau (z(v)^p-1)x)|_{v= \varpi}\right) = 0$ for $\kappa \not\in \mu(J_\mathrm{max})$;
	\item[(V2)] $c\left(E^{\operatorname{AH}}(b v^{p(p^f-1)/(p-1)})|_{v=\varpi}, E^{\operatorname{AH}}((v^{\Omega_{\tau}})_\tau (z(v)^p-1)x)|_{v=\varpi} \right) = 0$ for some $b \in l$ with $\operatorname{Tr}_{l/\mathbb{F}_p}b \neq 0$,
\end{enumerate}
for every $x \in l[[u]] \otimes_{\mathbb{F}_p} \overline{\mathbb{F}}_p$. In fact, in view of Proposition~\ref{maximalprop}, we can assume that for a fixed $\tau_0:k\rightarrow \overline{\mathbb{F}}_p$ we have either
\begin{itemize}
\item $x = \lambda_{\tau,\psi}$ whenever $\tau \circ \varphi^{-1} \in J_{\operatorname{max}}$, or
\item $x = \lambda_{\tau_0,\psi} v^{-\Omega_{\tau_0}}$ if $\psi=1$ and $-\Omega_{\tau_0} \in (p^f-1) \mathbb{Z}_{\geq 0}$.
\end{itemize}
Restricting to these $x$'s will simplify some of the computations.

\subsubsection*{Step 1: The vanishing in (V2) always occurs}
Since $\Omega_{\tau} > -p(p^f-1)/(p-1)$, we have
\begin{equation}\label{containing}
(v^{\Omega_{\tau}})_\tau (z(v)^p-1)x \in vl[[v]] \otimes_{\FFp} \overline{\mathbb{F}}_p.	
\end{equation}
Therefore, Corollary~\ref{cor-vos} computes the value in (V2) as the trace of the constant term of 
\begin{align*}
\sum_{m \geq 0} \left( (v^{\Omega_{\tau}})_\tau x \varphi^m \left( v\frac{d}{dv}(b v^{p(p^f-1)/(p-1)})\right) \right) 
- \sum_{m \geq 1} \left(\frac{1}{z(v)^p-1}\right)\left( b v^{p(p^f-1)/(p-1)} \varphi^m \left( v\frac{d}{dv}((v^{\Omega_{\tau}})_\tau (z(v)^p-1) x)\right) \right) 
\end{align*}
Using that $\Omega_{\tau} > -p(p^f-1)/(p-1)$, we see that this is an element of $vl[[v]] \otimes_{\FFp} \overline{\mathbb{F}}_p$. Therefore, the constant term vanishes.

\subsubsection*{Step 2: A formula for the value in (V1)}

Establishing the vanishing in (V1) will be more involved. We may assume for the remainder of the proof that $\kappa\notin \mu(J_\mathrm{max})$. We have already observed \eqref{containing} that $(v^{\Omega_{\tau}})_\tau (z(v)^p-1)x \in vl[[v]] \otimes_{\FFp} \overline{\mathbb{F}}_p$. We also have $v^{n_{\kappa'}} \lambda_{\kappa',\psi} \in vl[[v]] \otimes_{\FFp} \overline{\mathbb{F}}_p$. Therefore, Corollary~\ref{cor-vos} computes the value in (V1) as the trace of the constant term of
\begin{equation}\label{vanishingterms}
\begin{aligned}
\sum_{m \geq 0} \left( (\underbrace{v^{\Omega_{\tau}})_\tau x \varphi^m \left( v\frac{d}{dv}(v^{n_{\kappa}'} \lambda_{\kappa',\psi})\right)}_{(A_m)} \right) - \sum_{m \geq 1}\left( \underbrace{ \left(\frac{1}{z(v)^p-1}\right) v^{n_{\kappa}'} \lambda_{\kappa',\psi} \varphi^m \left( v\frac{d}{dv}((v^{\Omega_{\tau}})_\tau (z(v)^p-1) x)\right)}_{(B_m)} \right)
\end{aligned}
\end{equation}
for $\kappa \not\in \mu(J_{\operatorname{max}})$. To finish the proof of Proposition~\ref{DDR=us} we will establish the vanishing of the constant term of \eqref{vanishingterms} by considering the constant terms of each of the $(A_m)$ and $(B_m)$'s.

\subsubsection*{Step 3: Vanishing constant terms in $(A_m)$} First, supose $m\ge 0$ and assume $x = \lambda_{\tau,\psi}$ with $\tau \circ \varphi^{-1} \in J_{\operatorname{max}}$. Since $\varphi^m(\lambda_{\kappa',\psi}) = \lambda_{\kappa' \circ \varphi^{-m},\psi}$, we can write $(A_m)$ as
$$
n_{\kappa}' \lambda_{\tau,\psi} \lambda_{\kappa' \circ \varphi^{-m},\psi} v^{\Omega_\tau + p^m n_{\kappa}' }.  
$$
Therefore $(A_m)$ has a non-zero constant term only if
\begin{itemize}
\item $p^m n_{\kappa}' + \Omega_\tau = 0$
\item $\tau = \kappa' \circ \varphi^{-m}$
\end{itemize}
Writing $\tau ' = \tau \circ \varphi^{-1}$, we see that $\tau'\in J_\mathrm{max}$ and $\Omega_{\tau'} = p\Omega_\tau + (p^f-1)r_{\tau'}.$ Rewriting the two conditions above for $\tau'$, we obtain $p^{m+1}n_\kappa ' = -\Omega_{\tau'} +(p^f-1)r_{\tau'}$ and $\kappa' = \tau' \circ \varphi^{m+1}.$ Then it follows from \cite[3.6.7]{cegm17} that this implies $\kappa \in \mu(J_{\operatorname{max}})$, giving the desired contradiction. (To apply this proposition to our situation we note that in loc. cit. $\tau_0$ is fixed and $\tau_0 \circ \varphi^i$ is written $\tau_i$. Furthermore, $n_{\tau_0 \circ \varphi^i}'$ is written $n_{i}'$ and the value $\xi_i$ from loc. cit. is precisely $-\Omega_{\tau_0 \circ \varphi^i} + \delta_{J_\mathrm{max}}(\tau_0 \circ \varphi^i)r_{\tau_0 \circ \varphi^i}(p^f-1)$ with $\delta_{J_\mathrm{max}}$ the characteristic function for $J_\mathrm{max}$ on all embeddings.)

\begin{rem}\label{rem-wheremuisused}
This is where the precise definition of the shift function $\mu$ from the definition of $L_\sigma^{\operatorname{DDR}}(\chi_1,\chi_2)$ is used.
\end{rem}

The other case is when $\psi = 1$, $-\Omega_{\tau_0} \in (p^f-1)\mathbb{Z}_{\geq 0}$, and $x = \lambda_{\tau_0,\psi} v^{-\Omega_{\tau_0}}$. Then $(A_m)$ evaluates to $
n_{\kappa}' \lambda_{\tau_0,\psi} \lambda_{\kappa' \circ \varphi^{-m},\psi} v^{p^m n_{\kappa}'}$ which is clearly contained in $vl[[v]] \otimes_{\mathbb{F}_p} \overline{\mathbb{F}}_p$. Thus, the vanishing of the constant term here is clear.

\subsubsection*{Step 4: Vanishing residues of $(B_m)$}

Assume $m\ge 1$ and $x = \lambda_{\tau,\psi}$ with $\tau \circ \varphi^{-1} \in J_{\operatorname{max}}$. Since $z(v)^p-1$ is a $p$-th power in $l[[v]]$, we have $\frac{d}{dv}(z(v)^p-1) = 0$ in $l[[v]]$. Therefore, we can rewrite $(B_m)$ as
$$
\Omega_\tau  \lambda_{\kappa',\psi} \lambda_{\tau \circ \varphi^{-m},\psi} (z(v)^p-1)^{p^m-1} v^{n_{\kappa}' + p^m\Omega_\tau}.
$$
We may assume $\kappa' = \tau \circ \varphi^{-m}$ since $\lambda_{\kappa',\psi} \lambda_{\tau \circ \varphi^{-m},\psi}$ is zero otherwise. Since $z(v)^p-1$ has $v$-adic valuation $p(p^f-1)/(p-1)$, it suffices to show that
\begin{equation}\label{ineqwemustshow}
(p^m - 1)\left( \frac{p(p^f-1)}{p-1} \right) + n_\kappa ' + p^m\Omega_\tau > 0.
\end{equation}
Using that $n_{\kappa}' \geq (p^f-1)/(p-1)$ (see \cite[3.6.4]{cegm17}) and $\Omega_\tau > -p(p^f-1)/(p-1)$, we have 
\begin{equation*}
(p^m - 1)\left( \frac{p(p^f-1)}{p-1} \right) + n_{\kappa'} +p^m \Omega_\tau > \left(\frac{p^f-1}{p-1}\right) \left( p(p^m-1) +1 - p^{m+1}  \right) = -(p^f-1).
\end{equation*}
On the other hand, since $\kappa' = \tau \circ \varphi^{-m}$, it follows that
$$
\omega_{\kappa'}^{n_{\kappa}'} = \omega_{\tau}^{-\Omega_\tau} = \omega_{\kappa'}^{-p^m\Omega_\tau}
$$
and so $n_{\kappa}' + p^m\Omega_\tau \equiv 0$ modulo $p^f-1$. This, combined with the previous inequality, implies
\begin{equation}\label{ineq}
(p^m - 1)\left( \frac{p(p^f-1)}{p-1} \right) + n_{\kappa}' +p^m \Omega_\tau \geq 0.
\end{equation}
Thus, to prove \eqref{ineqwemustshow} we only have to show we cannot have an equality in \eqref{ineq} when $m\geq 1$. However, equality would imply $n_\kappa'$ is divisible by $p$ and this is not the case (for example, by \cite[3.6.1]{cegm17}).

The other possibility is when $\psi =1$, $-\Omega_{\tau_0} \in (p^f-1)\mathbb{Z}_{\geq 0}$, and $x = \lambda_{\tau_0,\psi} v^{-\Omega_{\tau_0}}$. However, in this case $(B_m)$ evaluates to zero so there is nothing to compute.

\printbibliography

\end{document}